\documentclass[a4paper,10pt, oneside]{amsart}
\usepackage[utf8]{inputenc}
\usepackage{url}
\usepackage[hidelinks]{hyperref}
\usepackage{pdfpages}
\usepackage[DIV=10]{typearea}

\usepackage{microtype}
\usepackage{amssymb,amsmath,amsopn,amsxtra,amsthm,amsfonts}
\usepackage{MnSymbol}
\usepackage[mathcal]{euscript}
\let\mathscr\mathcal
\usepackage[bb=px]{mathalfa}

\DeclareFontFamily{U}{min}{}
\DeclareFontShape{U}{min}{m}{n}{<-> udmj30}{}

\usepackage{enumitem}
\setlist[enumerate,1]{label={(\arabic*)},itemsep=\parskip} 
\setlist[itemize,1]{itemsep=\parskip} 
\newlist{thmlist}{enumerate}{2}
\setlist[thmlist,1]{label={\em(\roman*)},ref={(\roman*)},%
  itemsep=\parskip,leftmargin=*,align=left}
\setlist[thmlist,2]{label={\em(\alph*)},ref={(\alph*)},%
  itemsep=\parskip,leftmargin=*,align=left,topsep=0.1cm}
\newlist{defnlist}{enumerate}{2}
\setlist[defnlist,1]{label={(\roman*)},ref={(\roman*)},itemsep=\parskip,%
  leftmargin=*,align=left}
\setlist[defnlist,2]{label={(\alph*)},ref={(\alph*)},itemsep=\parskip,%
  leftmargin=*,align=left,topsep=0.1cm}

 \newtheorem*{thm*}{Theorem}

\newtheorem{thm}[subsection]{Theorem}
\newtheorem{cor}[subsection]{Corollary}
\newtheorem{lem}[subsection]{Lemma}
\newtheorem{prop}[subsection]{Proposition}

\theoremstyle{plain}
\newtheorem{defn}[subsection]{Definition}
\newtheorem{rem}[subsection]{Remark}

\newtheorem{exam}[subsection]{Example}

\newtheorem{constr}[subsection]{Construction}

\renewcommand{\eqref}[1]{(\ref{#1})}

\numberwithin{equation}{subsection}

\usepackage{tikz}
\usetikzlibrary{matrix}
\usepackage{tikz-cd}

\usepackage{xspace}

\usepackage{xifthen}

\usepackage{xparse}

\usepackage{mathtools}

\newcommand{\nc}{\newcommand}
\nc{\renc}{\renewcommand}
\nc{\ssec}{\subsection}
\nc{\sssec}{\subsubsection}
\nc{\on}{\operatorname}
\nc{\term}[1]{#1\xspace}

\nc{\sA}{\ensuremath{\mathcal{A}}\xspace}
\nc{\sB}{\ensuremath{\mathcal{B}}\xspace}
\nc{\sC}{\ensuremath{\mathcal{C}}\xspace}
\nc{\sD}{\ensuremath{\mathcal{D}}\xspace}
\nc{\sE}{\ensuremath{\mathcal{E}}\xspace}
\nc{\sF}{\ensuremath{\mathcal{F}}\xspace}
\nc{\sG}{\ensuremath{\mathcal{G}}\xspace}
\nc{\sH}{\ensuremath{\mathcal{H}}\xspace}
\nc{\sI}{\ensuremath{\mathcal{I}}\xspace}
\nc{\sJ}{\ensuremath{\mathcal{J}}\xspace}
\nc{\sK}{\ensuremath{\mathcal{K}}\xspace}
\nc{\sL}{\ensuremath{\mathcal{L}}\xspace}
\nc{\sM}{\ensuremath{\mathcal{M}}\xspace}
\nc{\sN}{\ensuremath{\mathcal{N}}\xspace}
\nc{\sO}{\ensuremath{\mathcal{O}}\xspace}
\nc{\sP}{\ensuremath{\mathcal{P}}\xspace}
\nc{\sQ}{\ensuremath{\mathcal{Q}}\xspace}
\nc{\sR}{\ensuremath{\mathcal{R}}\xspace}
\nc{\sS}{\ensuremath{\mathcal{S}}\xspace}
\nc{\sT}{\ensuremath{\mathcal{T}}\xspace}
\nc{\sU}{\ensuremath{\mathcal{U}}\xspace}
\nc{\sV}{\ensuremath{\mathcal{V}}\xspace}
\nc{\sW}{\ensuremath{\mathcal{W}}\xspace}
\nc{\sX}{\ensuremath{\mathcal{X}}\xspace}
\nc{\sY}{\ensuremath{\mathcal{Y}}\xspace}
\nc{\sZ}{\ensuremath{\mathcal{Z}}\xspace}

\nc{\bA}{\ensuremath{\mathbf{A}}\xspace}
\nc{\bB}{\ensuremath{\mathbf{B}}\xspace}
\nc{\bC}{\ensuremath{\mathbf{C}}\xspace}
\nc{\bD}{\ensuremath{\mathbf{D}}\xspace}
\nc{\bE}{\ensuremath{\mathbf{E}}\xspace}
\nc{\bF}{\ensuremath{\mathbf{F}}\xspace}
\nc{\bG}{\ensuremath{\mathbf{G}}\xspace}
\nc{\bH}{\ensuremath{\mathbf{H}}\xspace}
\nc{\bI}{\ensuremath{\mathbf{I}}\xspace}
\nc{\bJ}{\ensuremath{\mathbf{J}}\xspace}
\nc{\bK}{\ensuremath{\mathbf{K}}\xspace}
\nc{\bL}{\ensuremath{\mathbf{L}}\xspace}
\nc{\bM}{\ensuremath{\mathbf{M}}\xspace}
\nc{\bN}{\ensuremath{\mathbf{N}}\xspace}
\nc{\bO}{\ensuremath{\mathbf{O}}\xspace}
\nc{\bP}{\ensuremath{\mathbf{P}}\xspace}
\nc{\bQ}{\ensuremath{\mathbf{Q}}\xspace}
\nc{\bR}{\ensuremath{\mathbf{R}}\xspace}
\nc{\bS}{\ensuremath{\mathbf{S}}\xspace}
\nc{\bT}{\ensuremath{\mathbf{T}}\xspace}
\nc{\bU}{\ensuremath{\mathbf{U}}\xspace}
\nc{\bV}{\ensuremath{\mathbf{V}}\xspace}
\nc{\bW}{\ensuremath{\mathbf{W}}\xspace}
\nc{\bX}{\ensuremath{\mathbf{X}}\xspace}
\nc{\bY}{\ensuremath{\mathbf{Y}}\xspace}
\nc{\bZ}{\ensuremath{\mathbf{Z}}\xspace}

\nc{\dA}{\ensuremath{\mathds{A}}\xspace}
\nc{\dB}{\ensuremath{\mathds{B}}\xspace}
\nc{\dC}{\ensuremath{\mathds{C}}\xspace}
\nc{\dD}{\ensuremath{\mathds{D}}\xspace}
\nc{\dE}{\ensuremath{\mathds{E}}\xspace}
\nc{\dF}{\ensuremath{\mathds{F}}\xspace}
\nc{\dG}{\ensuremath{\mathds{G}}\xspace}
\nc{\dH}{\ensuremath{\mathds{H}}\xspace}
\nc{\dI}{\ensuremath{\mathds{I}}\xspace}
\nc{\dJ}{\ensuremath{\mathds{J}}\xspace}
\nc{\dK}{\ensuremath{\mathds{K}}\xspace}
\nc{\dL}{\ensuremath{\mathds{L}}\xspace}
\nc{\dM}{\ensuremath{\mathds{M}}\xspace}
\nc{\dN}{\ensuremath{\mathds{N}}\xspace}
\nc{\dO}{\ensuremath{\mathds{O}}\xspace}
\nc{\dP}{\ensuremath{\mathds{P}}\xspace}
\nc{\dQ}{\ensuremath{\mathds{Q}}\xspace}
\nc{\dR}{\ensuremath{\mathds{R}}\xspace}
\nc{\dS}{\ensuremath{\mathds{S}}\xspace}
\nc{\dT}{\ensuremath{\mathds{T}}\xspace}
\nc{\dU}{\ensuremath{\mathds{U}}\xspace}
\nc{\dV}{\ensuremath{\mathds{V}}\xspace}
\nc{\dW}{\ensuremath{\mathds{W}}\xspace}
\nc{\dX}{\ensuremath{\mathds{X}}\xspace}
\nc{\dY}{\ensuremath{\mathds{Y}}\xspace}
\nc{\dZ}{\ensuremath{\mathds{Z}}\xspace}

\nc{\bbA}{\ensuremath{\mathbb{A}}\xspace}
\nc{\bbB}{\ensuremath{\mathbb{B}}\xspace}
\nc{\bbC}{\ensuremath{\mathbb{C}}\xspace}
\nc{\bbD}{\ensuremath{\mathbb{D}}\xspace}
\nc{\bbE}{\ensuremath{\mathbb{E}}\xspace}
\nc{\bbF}{\ensuremath{\mathbb{F}}\xspace}
\nc{\bbG}{\ensuremath{\mathbb{G}}\xspace}
\nc{\bbH}{\ensuremath{\mathbb{H}}\xspace}
\nc{\bbI}{\ensuremath{\mathbb{I}}\xspace}
\nc{\bbJ}{\ensuremath{\mathbb{J}}\xspace}
\nc{\bbK}{\ensuremath{\mathbb{K}}\xspace}
\nc{\bbL}{\ensuremath{\mathbb{L}}\xspace}
\nc{\bbM}{\ensuremath{\mathbb{M}}\xspace}
\nc{\bbN}{\ensuremath{\mathbb{N}}\xspace}
\nc{\bbO}{\ensuremath{\mathbb{O}}\xspace}
\nc{\bbP}{\ensuremath{\mathbb{P}}\xspace}
\nc{\bbQ}{\ensuremath{\mathbb{Q}}\xspace}
\nc{\bbR}{\ensuremath{\mathbb{R}}\xspace}
\nc{\bbS}{\ensuremath{\mathbb{S}}\xspace}
\nc{\bbT}{\ensuremath{\mathbb{T}}\xspace}
\nc{\bbU}{\ensuremath{\mathbb{U}}\xspace}
\nc{\bbV}{\ensuremath{\mathbb{V}}\xspace}
\nc{\bbW}{\ensuremath{\mathbb{W}}\xspace}
\nc{\bbX}{\ensuremath{\mathbb{X}}\xspace}
\nc{\bbY}{\ensuremath{\mathbb{Y}}\xspace}
\nc{\bbZ}{\ensuremath{\mathbb{Z}}\xspace}

\nc{\mrm}[1]{\ensuremath{\mathrm{#1}}\xspace}
\nc{\mit}[1]{\ensuremath{\mathit{#1}}\xspace}
\nc{\mbf}[1]{\ensuremath{\mathbf{#1}}\xspace}
\nc{\mcal}[1]{\ensuremath{\mathcal{#1}}\xspace}
\nc{\msc}[1]{\ensuremath{\mathscr{#1}}\xspace}
\nc{\mfr}[1]{\ensuremath{\mathfrak{#1}}\xspace}

\renc{\bar}[1]{\overline{#1}}
\nc{\sub}{\subset}
\nc{\too}{\longrightarrow}
\nc{\hook}{\hookrightarrow}
\nc*{\hooklongrightarrow}{\ensuremath{\lhook\joinrel\relbar\joinrel\rightarrow}}
\nc{\hooklong}{\hooklongrightarrow}
\nc{\twoheadlongrightarrow}{\relbar\joinrel\twoheadrightarrow}
\nc{\shiso}{\approx}
\nc{\isoto}{\xrightarrow{\sim}}
\nc{\isofrom}{\xleftarrow{\sim}}
\renc{\ge}{\geqslant}
\renc{\le}{\leqslant}

\nc{\id}{\mathrm{id}}

\DeclareMathOperator{\Hom}{\on{Hom}}
\nc{\uHom}{\underline{\smash{\Hom}}}
\DeclareMathOperator{\Maps}{\on{Maps}}

\DeclareMathOperator{\End}{\on{End}}

\nc{\uEnd}{\underline{\smash{\End}}}

\renc{\lim}{\varprojlim}
\makeatletter
\newcommand{\colim@}[2]{%
  \vtop{\m@th\ialign{##\cr
    \hfil$#1\operator@font colim$\hfil\cr
    \noalign{\nointerlineskip\kern1.5\ex@}#2\cr
    \noalign{\nointerlineskip\kern-\ex@}\cr}}%
}
\newcommand{\colim}{%
  \mathop{\mathpalette\colim@{\rightarrowfill@\textstyle}}\nmlimits@
}
\makeatother

\nc{\Cofib}{\on{Cofib}}
\nc{\Fib}{\on{Fib}}
\nc{\initial}{\varnothing}
\nc{\op}{\mathrm{op}}

\nc{\Spc}{\mrm{Spc}}
\nc{\Spt}{\mrm{Spt}}
\nc{\Spec}{\on{Spec}}
\nc{\Stk}{\mrm{Stk}}
\nc{\Sch}{\mrm{Sch}}
\nc{\aff}{\mrm{aff}}
\nc{\A}{\mbf{A}}
\renc{\P}{\mbf{P}}
\nc{\cl}{{\mrm{cl}}}
\nc{\bDelta}{\mathbf{\Delta}}
\nc{\un}{\mathbf{1}}
\nc{\Tot}{\on{Tot}}
\nc{\Cech}{\textnormal{\v{C}}}
\nc{\Mod}{\mrm{Mod}}
\nc{\Qcoh}{\on{Qcoh}}
\nc{\free}{\mrm{free}}
\nc{\perf}{\mrm{perf}}
\nc{\aperf}{\mrm{aperf}}
\nc{\coh}{\mrm{coh}}
\nc{\Einfty}{{\sE_\infty}}
\nc{\modmod}{/\!\!/}
\nc{\heart}{\heartsuit}
\nc{\proj}{\mrm{proj}}
\nc{\K}{\on{K}}
\nc{\G}{\on{G}}
\nc{\GL}{\on{GL}}
\nc{\BGL}{\on{BGL}}
\nc{\M}{\on{M}}
\nc{\KH}{\on{KH}}
\nc{\Alg}{\on{Alg}}
\nc{\CAlg}{\on{CAlg}}
\nc{\cn}{\mrm{cn}}
\nc{\hw}{\mrm{Hw}}
\nc{\htt}{\mrm{Ht}}
\nc{\Fun}{\on{Fun}}
\nc{\Funadd}{\on{Fun}_{\mrm{add}}}
\nc{\Funex}{\on{Fun}_{\mrm{ex}}}
\nc{\Ind}{\on{Ind}}
\nc{\Pro}{\on{Pro}}
\nc{\Kar}{\on{Kar}}
\nc{\Obj}{\on{Obj}}

\nc{\scr}{\term{simplicial commutative ring}}
\nc{\scrs}{\term{simplicial commutative rings}}

\nc{\Einfring}{\term{$\Einfty$-ring}}
\nc{\Einfrings}{\term{$\Einfty$-rings}}

\nc{\Ering}{\term{$\sE_1$-ring}}
\nc{\Erings}{\term{$\sE_1$-rings}}

\nc{\inftyCat}{\term{$\infty$-category}}
\nc{\inftyCats}{\term{$\infty$-categories}}

\nc{\inftyTop}{\term{$\infty$-topos}}
\nc{\inftyTops}{\term{$\infty$-toposes}}

\nc{\inftyGrpd}{\term{$\infty$-groupoid}}
\nc{\inftyGrpds}{\term{$\infty$-groupoids}}
\title{}

\begin{document}

\title{$\mathbb{A}^1$-invariance of localizing invariants}
\author{Vladimir Sosnilo}
\address{
M309\\
Universit{\"a}t Regensburg\\
Universit{\"a}tsstra{\ss}e 31 \\
93053 Regensburg\\
Germany
}
\email{\href{mailto:vsosnilo@gmail.com}{vsosnilo@gmail.com}}

\bibliographystyle{alphamod}

\begin{abstract}
Weibel proved that $p$-inverted K-theory is 
$\mathbb{A}^1$-invariant on $\mathbb{F}_p$-schemes and K-theory with $\mathbb{Z}/p$-coefficients 
is $\mathbb{A}^1$-invariant on $\mathbb{Z}[\frac{1}{p}]$-schemes. 
We extend this result to all finitary localizing invariants of small stable $\infty$-categories. 
Along the way we study the Frobenius and Verschiebung endofunctors defined by Tabuada 
and provide a categorical version of Stienstra's projection formula.
\end{abstract}

\maketitle

\setcounter{tocdepth}{1}
\tableofcontents

\section*{Introduction}

In the very foundation of motivic homotopy theory lies the idea of enforcing the affine line to be 
contractible, 
hence imposing the {\it $\mathbb{A}^1$-invariance property}\footnote{also called {\it homotopy invariance}} on all cohomology theories of one's interest. 
Concretely, a cohomology theory\footnote{i.e. a contravariant functor from the category of schemes} $E$ is said to be {\it $\mathbb{A}^1$-invariant} if the projection map induces an equivalence
\[
E(X) \simeq E(X \times \mathbb{A}^1)
\]
for all schemes $X$.
This is not as restrictive as it sounds at first, because many important cohomology theories satisfy 
$\mathbb{A}^1$-property automatically on regular schemes. For example, this is true for K-theory, 
Chow groups, Morel-Levine's algebraic cobordism in characteristic 0 (see \cite{Levine:2007}). 

However, in general one is also interested in studying singular schemes. 
In this context K-theory is known not to be $\mathbb{A}^1$-invariant---moreover, conjecturally 
$\mathbb{A}^1$-invariance of K-theory on a scheme detects its regularity (\cite[Conjecture]{Vorst1979}, \cite{Vorst_conj_GH}). 
As another example, the
derived algebraic cobordism of Annala is not $\mathbb{A}^1$-invariant (see \cite[Theorem~236]{annala_thesis}). 
This motivates the quest for {\it non-$\mathbb{A}^1$-invariant motivic homotopy theory}. 
\ssec{Motivic theories} One approach is 
to replace $\mathbb{A}^1$-invariance with the property of satisfying the 
projective bundle formula. This is done 
via  {\it pbf-local sheaves with transfers} 
and {\it pbf-local $\mathbb{P}^1$-spectra} in the recent works of Annala and Iwasa (see 
\cite{annala2022motivic} and \cite{AnnalaIwasa}) and also their joint work with Hoyois \cite{annala2023algebraic}. 
Another more na{\"i}ve approach is to study {\it localizing invariants} (aka {\it noncommutative motives}) whose 
definition does not refer to any geometric properties at all but is nevertheless sensible enough 
to reflect many of them. Recall that a localizing invariant is a functor $E$ from the $\infty$-category of small idempotent complete stable $\infty$-categories $\mrm{Cat}^{\mrm{perf}}_\infty$ to spectra $\Spt$ satisfying: 
\begin{itemize}
\item for any exact sequence $\sA \to \sB \to \sC$ 
\[
E(\sA) \to E(\sB) \to E(\sC)
\]
is a fiber sequence.
\end{itemize}
We say that a localizing invariant $E$ is finitary if it preserves filtered colimits\footnote{In \cite{blumberg2013universal} this property was included in the definition of a 
localizing invariant.}. Any localizing 
invariant gives rise to a pbf-local $\mathbb{P}^1$-spectrum. Given a commutative ring $k$ we will also consider localizing invariants 
of small $k$-linear stable $\infty$-categories. We call them $k$-localizing invariants.

Many ``motivic'' properties are shared by pbf-local $\mathbb{P}^1$-spectra and localizing invariants. 
To list a few, the Nisnevich descent \cite[Corollary~4.1.2]{BKRS}, the pro-cdh-descent for noetherian 
schemes \cite[Theorem~A.8]{Land_2019}, \cite[Theorem~C]{BKRS} hold for all localizing invariants. 
In the setting of pbf-local $\mathbb{P}^1$-spectra the motives of Grassmanians $\mrm{Gr}_n(S)$ and 
classifying stacks of vector bundles 
$\mathcal{V}ect_{n,S}$ can be computed as in the 
$\mathbb{A}^1$-invariant theory \cite[Corollary~4.2.5]{annala2022motivic}. Annala and Iwasa also prove that 
K-theory can be obtained via a direct 
analogue of Snaith's construction \cite[Theorem~0.1.1]{annala2022motivic} (the $\mathbb{A}^1$-invariant analogue was previously proved in \cite{Gepner2009OnTM}). 
\ssec{Back to {$\mathbb{A}^1$}-invariance}
In many examples we have some control over the failure of $\mathbb{A}^1$-invariance. 
In particular, this is illustrated by the following result of Weibel:
\begin{thm}[{\cite[Corollary~3.3]{weibel-nk}}]\label{thm:weibelKA1}
$K(-)[\frac{1}{p}]$ is $\mathbb{A}^1$-invariant on $\mathbb{F}_p$-schemes. 
Moreover, $K(-)/l$ is $\mathbb{A}^1$-invariant on $\mathbb{Z}[\frac{1}{l}]$-schemes.
\end{thm}

Recently, Elmanto and Morrow have constructed a non-$\mathbb{A}^1$-invariant version of motivic 
cohomology over schemes of characteristic $p$ (see \cite{MotivicCohomology}). 
It is a well-behaved extension of the usual (meaning the one defined in \cite{SpitzweckHZ}) 
$\mathbb{A}^1$-invariant 
motivic cohomology to non-regular schemes. 
The two theories coincide after inverting $p$, so the direct analogue of the above theorem  holds. 
It follows from \cite[Theorem~4.22]{annala2021algebraic} that the {\it homological} version of the derived 
algebraic cobordism is also $\mathbb{A}^1$-invariant away from the characteristic. 
So, it is not unreasonable to expect a similar theorem to hold even for a larger class of
non-$\mathbb{A}^1$-invariant motivic theories. 
We do not know whether this is true for pbf-local $\mathbb{P}^1$-spectra of 
Annala-Iwasa under some assumptions, or for the derived algebraic cobordism. However, we prove that 
this is always true for all finitary localizing invariants:

\begin{thm}[Theorem~\ref{thm:main_result_proof}]\label{thm:main_result}
Let $k$ be a connective ring spectrum with $\pi_0(k)$ of characteristic $p$ and $E$ a $k$-localizing invariant 
which is finitary. 
Then 
$E(-)[\frac{1}{p}]$ is $\mathbb{A}^1$-invariant.
\end{thm}

We also generalize the other part of Weibel's result:

\begin{thm}[Theorem~\ref{thm:main_result_2_proof}]\label{thm:main_result_2}
Let $k$ be a connective $\mathbb{S}[\frac{1}{l}]$-algebra spectrum and $E$ a finitary $k$-localizing 
invariant. 
Then 
$E(-)/l$ is  $\mathbb{A}^1$-invariant.
\end{thm}

In particular, this implies that Theorem~\ref{thm:weibelKA1} holds for all spectral schemes. 
In the context of dg-categories for $E$ being the functor of nonconnective K-theory these theorems 
have been proved in \cite{Tabuada_Kleinian}. 
We note that the assumption of being finitary is necessary: for example, it follows from 
\cite[Theorem~3.1.1]{THH_TC_A1} that $TC\otimes \mathbb{Q}(-)$ is not $\mathbb{A}^1$-invariant. However, the 
author expects that Theorem~\ref{thm:main_result_2} holds more generally.

\ssec{Outline of the proof}

We follow the same steps as the original Weibel's proof \cite{weibel-nk}. 
Given a $k$-localizing invariant define
\[
NE(\sC) := \Fib (E(\sC \otimes_k \mbf{Perf}_{k[t]}) \to E(\sC)).
\]
Similar to how it is done in \cite{Tabuada_Kleinian}, we establish an action of {\it the ring of rational 
Witt vectors} $\mathbb{W}_0(k)$ on the homotopy groups of $NE(\sC)$. The main ingredients going into 
defining this action are Almkvist's isomorphism \cite{ALMKVIST1974375} 
\[
\mathbb{Z} \oplus \mathbb{W}_0(k) \simeq K_0(\mathcal{E}nd(\mbf{Perf}_k))
\]
and an explicit isomorphism 
\[
\pi_iE(\sC) \oplus \pi_{i+1}NE(\sC) \simeq \pi_iE(\mathcal{N}il(\sC)).
\]
Here
$\mathcal{E}nd(\sC)$ is the stable $\infty$-category of endomorphisms of $\sC$ 
and $\mathcal{N}il(\sC)$ is the 
subcategory of nilpotent endomorphisms. 
Having these, the action is induced by the tensor product functor 
\[
\mathcal{E}nd(\mbf{Perf}_k) \times \mathcal{N}il(\sC) \to \mathcal{N}il(\sC).
\]
Our main results then follow from a thorough analysis of the action, granted by the following steps: 
\begin{enumerate}
\item
There are Frobenius and Verschiebung endofunctors 
$\mathcal{F}_n, \mathcal{V}_n$ on $\mathcal{N}il(\sC)$ as well as $\mathcal{E}nd(\sC)$ 
(see \cite[Section~5]{blumberg2015witt}). We provide a different way of defining them, which 
allows to show that they are nicely compatible with the ring structure of $\mathbb{W}_0(k)$ and 
satisfy the projection formula with respect to the action of $\mathbb{W}_0(k)$ on $NE(\sC)$ (see Lemma~\ref{lem:proj_formula}). For $\sC=\mbf{Perf}_k$ 
they induce the classical Frobenius and Verschiebung operations on 
$K_0(\mathcal{E}nd(\mbf{Perf}_k)) \simeq \mathbb{Z} \oplus \mathbb{W}(\pi_0(k))$. 
This part allows us to represent  
the action of $p\in \mathbb{W}_0(k)$ on $NE(\sC)$ by an explicit endofunctor of $\mathcal{N}il(\sC)$. 

\item
The action of $\mathbb{W}_0(k)$ on the (discrete) abelian group $\pi_iE(\mathcal{N}il(\sC))$ is continuous. 
For $E=K$ this follows from the construction of K-theory. In general, we introduce a filtration 
\[
\cdots \to \mathcal{N}il^{\le n}(\sC) \to \mathcal{N}il^{\le n+1}(\sC) \to \cdots 
\]
of $\mathcal{N}il(\sC)$ such that 
\[
\colim\limits_n \mathcal{N}il^{\le n}(\sC) \simeq \mathcal{N}il(\sC)
\]
and so that the action is continuous on the images of 
all $\pi_iE(\mathcal{N}il^{\le n}(\sC))$ by construction (see (\ref{ssec:filtration}) and (\ref{sssec:filtration})). 
This allows us to extend the action of $\mathbb{W}_0(k)$ to an action of the full ring of Witt vectors 
$\mathbb{W}(k)$ and 
deduce the main results from its properties. 
\end{enumerate}
The constructions in each of the steps are interesting on their own and appear to be completely new.

\ssec{Structure of the paper}
In Section~\ref{section:locinv} we remind the reader all they need to know about localizing invariants. 
In Section~\ref{section:endcats} we introduce the $\infty$-categories of endomorphisms and of 
nilpotent endomorphisms. In Section~\ref{section:nilend_a1inv} we study 
nil-theories $NE(-)$ that measure the failure of $\mathbb{A}^1$-invariance of $E$.  
We show that $NE(\sC)$ can be recovered from the values of $E$ on the 
$\infty$-category of nilpotent endomorphisms in $\sC$ for any 
localizing invariant $E$. In Section~\ref{section:characteristic polynomial} we 
recall the classical results about the ring of big Witt vectors. We then study the Frobenius and 
Verschiebung endofunctors of the $\infty$-categories of endomorphisms and relate them to the classical 
Frobenius and Verschiebung operations on the ring of Witt vectors. Finally, in 
Section~\ref{section:main_result} 
we prove the main result.


\ssec*{Notations \& Conventions} 
\begin{itemize}
\item 
Throughout the paper we fix a connective commutative ring spectrum $k$. 
We denote by $k[t]$ the ring spectrum $k \otimes \Sigma^\infty_+ \mathbb{N}$. Compatibly, 
we will denote by $\mathbb{A}^1_k$ and $\mathbb{P}^1_k$ the {\it flat} affine line and the {\it flat} 
projective line.

\item 
All spectral $k$-schemes are assumed quasi-compact, quasi-separated and build out of connective $\mathbb{E}_\infty$-algebras. 
They form an $\infty$-category 
$\mrm{Sch}_k$.

\item 
Let $X$ be a spectral scheme and $Y$ be a closed subscheme of $X$. We denote by $\mbf{Qcoh}_X$ the 
$\infty$-category of quasi-coherent sheaves on $X$ and by $\mbf{Qcoh}_{Y,X}$ the subcategory of 
sheaves supported on $Y$. 
We denote by $\mbf{Perf}_X$ the subcategory of perfect complexes over $X$. 

\item 
For an $\mathbb{E}_\infty$-ring $R$ we denote by $\mbf{Mod}_R$ the $\infty$-category of $R$-modules in 
spectra, $\mbf{Perf}_R$ is the subcategory of perfect complexes. 

\item
For an $\infty$-category $\sC$ we denote by $\sC^{\simeq}$ its maximal groupoid and by $\sC^\omega$ 
the subcategory of compact objects.

\item
For $\infty$-categories $\sC,\sD$, $\mrm{Fun}(\sC,\sD)$ is the $\infty$-category of functors between them. If $\sC$ and $\sD$ admit finite colimits $\mrm{Fun}^{\mrm{ex}}(\sC,\sD)$ is the subcategory of those functors 
that preserve finite colimits. 
If $\sC$,$\sD$ are also $k$-linear for some commutative ring spectrum $k$, we denote by 
$\mrm{Fun}^{\mrm{ex}}_{k}(\sC,\sD)$ the $\infty$-category of $k$-linear functors that preserve finite colimits. 

\item 
Given a small stable $\infty$-category $\sC$ we use the following notation for the 
colimit-completion: 
\[
\widehat{\sC} := \mrm{Fun}^{\mrm{ex}}(\sC^{\mrm{op}}, \Spt).
\]
This is a stable (presentable) $\infty$-category and it admits a fully faithful Yoneda embedding $\sC \to \widehat{\sC}$.

\item 
For a stable $\infty$-category $\sC$ and objects $x, y \in \sC$ we denote by 
\[
\mrm{maps}_\sC(x,y)
\]
the mapping spectrum in $\sC$, while we keep the notation
\[
\Maps_\sC(x,y)
\]
for the mapping space.

\item
Given an $\mathbb{E}_1$-ring $R$, $\mbf{Mod}_R$ is the $\infty$-category of {\it right} $R$-modules, $\mbf{Perf}_R$ is the subcategory of perfect complexes.

\item 
We denote the $\infty$-category of small $k$-linear idempotent complete 
stable $\infty$-categories by $\mrm{Cat}^{\mrm{perf}}_{\infty,k}$. 
\end{itemize}

\ssec*{Acknowledgements}
The results of the paper emerged after several busy conversations with Ryomei Iwasa, Marc Hoyois and Niklas 
Kipp. The author thanks them for sharing their ideas, wishes and insights. 
Toni Annala pointed out to the author that the Frobenius and Verschiebung endofunctors 
$\mathcal{F}_n$, $\mathcal{V}_n$ are induced by the 
pullback and pushforward functors along the maps $k[t] \stackrel{t\mapsto t^n}\longrightarrow k[t]$. 
The author also wants to thank Denis-Charles Cisinski for a discussion about non-finitary 
localizing invariants, Alexander Efimov for sharing his wisdom about localizing motives and 
the anonymous referee for their valuable suggestions. 
Finally, I want to thank Maxime Ramzi for pointing out a mistake in 
Lemma~\ref{lem:proj_formula} and many incredibly useful discussions about questions related to the paper. 

The author gladly acknowledges that he is a member of the SFB 1085 ``Higher Invariants''  
funded by the Deutsche Forschungsgesellschaft (DFG).

\section{Reminder on noncommutative motives and localizing invariants}\label{section:locinv}

In this section we recall some basic notions and facts from the theory of localizing invariants. 
We only discuss those results that are relevant for this paper. 
For a general theory of localizing invariants we refer to \cite{blumberg2013universal}. We start from 
basic definitions:

\begin{defn}
We say that a diagram in $\mrm{Cat}^{\mrm{perf}}_{\infty,k}$
\[
\sA \stackrel{f}\to \sB \stackrel{g}\to \sC
\]
is an {\bf exact sequence} of small $k$-linear idempotent complete stable $\infty$-categories if 
$f$ is fully faithful, the composite $g\circ f$ is trivial and the induced functor 
\[
\sB/\sA \to \sC
\]
is an equivalence up to idempotent completion.
\end{defn}

\begin{defn}
Let $\sE$ be a stable $\infty$-category. A functor 
\[
E: \mrm{Cat}^{\mrm{perf}}_{\infty,k} \to \sE
\]
is said to be a $k$-{\bf localizing invariant} if 
for any exact sequence $\sA \to \sB \to \sC$ of $k$-linear small idempotent complete stable 
$\infty$-categories the sequence
\[
E(\sA) \to E(\sB) \to E(\sC)
\]
is a fiber sequence.
\end{defn}

\sssec{}
Any $k$-localizing invariant gives rise to a cohomology theory on all $k$-schemes via assigment
\[
X \mapsto E(\mbf{Perf}_X)
\]
for any $X \in \mrm{Sch}_k.$
We abuse the notation and write 
\[
E(X) := E(\mbf{Perf}_X).
\]
We also write 
\[
E(R) := E(\mrm{Spec}(R)) 
\]
when $R$ is a $k$-algebra.

\sssec{}\label{rem:additivity}
One of the most useful elementary properties of localizing invariants 
is the following result, sometimes called the additivity theorem. 
For a localizing invariant $E$ and a fiber sequence 
\[
f \to g \to h
\]
in $\mrm{Fun}^{\mrm{ex}}_k(\sC, \sD)$ for any small $k$-linear idempotent complete 
stable $\infty$-categories $\sC,\sD$ one has an equivalence of maps of spectra
\[
E(g) \simeq E(f) + E(h) : E(\sC) \to E(\sD).
\]

\begin{exam}
K-theory, topological Hochschild homology, topological Cyclic homology, TR, Blanc's topological K-theory, Selmer K-theory are $k$-localizing invariants for any $k$. 
Homotopy K-theory is a $k$-localizing invariant if $k$ is a $\mathbb{Z}$-algebra.
\end{exam}

\begin{prop}\label{prop:pbf}
For any spectral scheme $X$ the pullback along the projection $\mathbb{P}^1_X \stackrel{p}\to X$ and 
the pushforward along any section 
$X\stackrel{i}\to \mathbb{P}^1_X$
induce an equivalence 
\[
E(\mbf{Perf}_X) \oplus E(\mbf{Perf}_X) \stackrel{\begin{pmatrix}i_* & p^* \end{pmatrix}}\simeq 
E(\mbf{Perf}_{\mathbb{P}^1_X})
\]
\end{prop}
\begin{proof}
Follows from \cite[Theorem~7.2.2.1]{SAG} (see also \cite[Theorem~B]{khan2018algebraic}).
\end{proof}

\begin{exam}
Motivic cohomology of a fixed weight $m$ does not extend to a $k$-localizing invariant for any $k$ because 
there is no equivalence as in Proposition~\ref{prop:pbf}.
\end{exam}

\begin{prop}\label{prop:P1_excision}
Let $E$ be a $k$-localizing invariant $E$. 
Then the square of spectra
\[
\begin{tikzcd}
E(\mathbb{P}^1_k) \arrow[r]\arrow[d] & E(\mathbb{A}^1_k)\arrow[d] \\
E(\mathbb{A}^1_k) \arrow[r] & E(\mathbb{G}_{m,k})
\end{tikzcd}
\]
is cartesian.
\end{prop}
\begin{proof}
The case of $k$ being a simplicial commutative ring follows from \cite[Theorem~4.1.1]{BKRS}. 
The proof in general is the same, but we include it for the reader's convenience. 

The diagram in question is obtained via applying $E$ to the square
\[
\begin{tikzcd}
\mbf{Perf}_{\mathbb{P}^1_k} \arrow[r]\arrow[d] & \mbf{Perf}_{\mathbb{A}^1_k}\arrow[d] \\
\mbf{Perf}_{\mathbb{A}^1_k} \arrow[r] & \mbf{Perf}_{\mathbb{G}_{m,k}},
\end{tikzcd}
\]
so by \cite[Theorem~18]{tamme-excision} it suffices to check that the square is excisive in the sense of 
\cite[Definition~14]{tamme-excision}. 
By \cite[Theorem~7.2.3.1]{SAG} we have a morphism of localization sequences 
\[
\begin{tikzcd}
\mbf{Qcoh}_{0, \mathbb{P}^1_k}\arrow[d,"\simeq"] \arrow[r] & \mrm{Ind}(\mbf{Perf}_{\mathbb{P}^1_k}) \arrow[r]\arrow[d] & \mrm{Ind}(\mbf{Perf}_{\mathbb{A}^1_k})\arrow[d] \\
\mbf{Qcoh}_{0, \mathbb{A}^1_k}\arrow[r] & \mrm{Ind}(\mbf{Perf}_{\mathbb{A}^1_k}) \arrow[r] & \mrm{Ind}(\mbf{Perf}_{\mathbb{G}_{m,k}}).
\end{tikzcd}
\]
In particular, the right horizontal maps are localizations and the right square is pullback and the 
original square in question is excisive.
\end{proof}

\ssec{}\label{ssec:sym_mon} Recall that $\mrm{Cat}^{\mrm{perf}}_{\infty, k}$ is symmetric monoidal in such a way that 
$\sC \otimes_k \sD$ corepresents functors $\sC \times \sD$
that are $k$-linear and exact in each variable. 
One can also see that $- \otimes_k \sD$ left adjoint to $\mrm{Fun}^{\mrm{ex}}_k(\sD, -)$
(see \cite[Section~4.1]{HSS}, also \cite[Section~3.1]{blumberg2013universal}).
In fact, one can canonically identify $\sC \otimes_k \sD$ with the subcategory of compact objects in 
$\mrm{Fun}^{\mrm{ex}}(\sC^{\mrm{op}}, \widehat{\sD})$ (see \cite[Corollary~4.11]{HSS}).

\begin{lem}\label{lem:tensor_prod_ex}
For any small $k$-linear idempotent complete stable $\infty$-category $\sC$
\[
- \otimes_k \sC : \mrm{Cat}^{\mrm{perf}}_{\infty, k} \to \mrm{Cat}^{\mrm{perf}}_{\infty, k}
\] 
preserves exact sequences and colimits.
\end{lem}
\begin{proof}
See Lemma~9.35 of \cite{blumberg2013universal}.
\end{proof}

\ssec{Non-commutative motives}
Tabuada, along with Blumberg, Cisinski and Gepner has developed a theory of non-commutative motives based 
on localizing invariants. 
The idea to define the $\infty$-category of motives of small idempotent complete stable $\infty$-categories 
as a universal localizing invariant comes quite naturally after 
one settles on the idea that localizing invariants are a good categorical analogue of cohomology theories. 

One expects such an $\infty$-category of motives to be constructed as a localization of the $\infty$-category of presheaves on 
$\mrm{Cat}^{\mrm{perf}}_{\infty}$ and there should be a ``motive'' functor 
\[
M:\mrm{Cat}^{\mrm{perf}}_{\infty} \to \mrm{Fun}^{\mrm{loc}}((\mrm{Cat}^{\mrm{perf}}_{\infty})^\mrm{op}, \Spt)
\]
such that 
$
E(\sC) \simeq \mrm{maps}(M(\sC), E) 
$
for any $E \in \mrm{Fun}^{\mrm{loc}}((\mrm{Cat}^{\mrm{perf}}_{\infty})^{\mrm{op}}, \Spt)$.
However, this is not easy to achieve because of set-theoretic difficulties, so instead, one considers 
{\bf finitary} localizing invariants.

\begin{defn}
Let $\sE$ be a cocomplete stable $\infty$-category. We say that a functor 
\[
E: \mrm{Cat}^{\mrm{perf}}_{\infty,k} \to \sE
\]
is {\bf finitary} if it preserves filtered colimits. For a  cocomplete stable $\infty$-category $\sD$ we denote the subcategory of $\mrm{Fun}(\mrm{Cat}^{\mrm{perf}}_{\infty}, \sD)$ consisting of finitary localizing invariants 
by $\mrm{Fun}^{\mrm{loc}, \mrm{fin}}(\mrm{Cat}^{\mrm{perf}}_{\infty, k}, \sD)$.
\end{defn}

\begin{exam}
K-theory, topological Hochschild homology, homotopy K-theory, Blanc's topological K-theory are finitary localizing invariants. 
\end{exam}

We denote the $\infty$-category of finitary $k$-localizing presheaves by $\mathcal{M}_{\mrm{loc}, k}$.

\begin{rem}
The $\infty$-category $\mrm{Cat}^{\mrm{perf}}_{\infty, k}$ is compactly generated by compact 
$k$-linear stable $\infty$-categories (see \cite[Proposition~4.7]{HSS} and \cite[Corollary~4.25]{blumberg2013universal}). We denote this subcategory by $\mrm{Cat}^{\mrm{perf},\omega}_{\infty, k}$. 
In particular, this implies that the restriction functor 
\[
\mrm{Fun}(\mrm{Cat}^{\mrm{perf}}_{\infty,k}, \sE) \to \mrm{Fun}(\mrm{Cat}^{\mrm{perf},\omega}_{\infty,k}, \sE)
\]
induces an equivalence between finitary functors on $\mrm{Cat}^{\mrm{perf}}_{\infty,k}$ and all functors on $\mrm{Cat}^{\mrm{perf},\omega}_{\infty,k}$. 
Moreover, every $k$-localization sequence is a filtered colimit of $k$-localization sequences between compact 
$k$-linear stable $\infty$-categories (\cite[Proposition~5.5]{HSS}). Therefore, we also have an equivalence 
\[
\mathcal{M}_{\mrm{loc}, k} \simeq \mrm{Fun}^{\mrm{loc}}((\mrm{Cat}^{\mrm{perf},\omega}_{\infty,k})^{\mrm{op}}, \Spt).
\]
\end{rem}

\begin{thm}[{\cite[Theorem~1.1]{blumberg2013universal}}]
The embedding functor 
\[
\mathcal{M}_{\mrm{loc}, k} \to \mrm{Fun}((\mrm{Cat}^{\mrm{perf},\omega}_{\infty, k})^\mrm{op}, \Spt)
\]
admits a left adjoint. 
In particular, after precomposing with the Yoneda embedding one obtains a functor
\[
\mathcal{U}_{\mrm{loc},k}: \mrm{Cat}^{\mrm{perf}}_{\infty, k} \to \mathcal{M}_{\mrm{loc}, k}.
\]
This functor is a universal finitary localizing invariant in the sense that it induces an equivalence 
\[
\mrm{Fun}^{\mrm{L}}(\mathcal{M}_{\mrm{loc}, k}, \sD) \simeq \mrm{Fun}^{\mrm{loc}, \mrm{fin}}(\mrm{Cat}^{\mrm{perf}}_{\infty}, \sD)
\]
for any stably presentable $\infty$-category $\sD$. 
In this category one has $E(\sC) \simeq \mrm{maps}(\mathcal{U}_{\mrm{loc},k}(\sC), E)$ 
for any $E \in \mathcal{M}_{\mrm{loc}, k}$.
\end{thm}

\begin{rem}
Although all the constructions and the arguments in \cite{blumberg2013universal} are done 
in the non-linear case, i.e. the case of localizing invariants, 
they work verbatim for  $k$-localizing invariants. See \cite[4,5]{HSS} for details.
\end{rem}

\sssec{}\label{rem:presentability}
By construction, the $\infty$-category $\mathcal{M}_{\mrm{loc}, k}$ is presentable. 
In fact, it admits a presentable symmetric monoidal enhancement such that 
\[
\mathcal{U}_{\mrm{loc},k}: \mrm{Cat}^{\mrm{perf}}_{\infty, k} \to \mathcal{M}_{\mrm{loc}, k} 
\]
is symmetric monoidal. 

\sssec{}\label{rem:A1localization}
Presentability of $\mathcal{M}_{\mrm{loc}, k}$ implies the existence of a symmetric monoidal 
Bousfield localization with respect to the arrow $\mrm{Perf}_{k} \to \mrm{Perf}_{k[t]}$
\[
L_{\mathbb{A}^1}:\mathcal{M}_{\mrm{loc}, k} \to \mathcal{M}_{\mrm{loc}, \mathbb{A}^1, k}.
\]
This allows to define the {\it $\mathbb{A}^1$-invariantization} 
\[
L_{\mathbb{A}^1}E(\sC) := E(L_{\mathbb{A}^1}\sC)
\] 
of any finitary localizing invariant $E$. 
If $\sC$ is a $\mathbb{Z}$-linear, $L_{\mathbb{A}^1}K$ is equivalent to the homotopy K-theory defined 
previously (see \cite[Definition~3.13]{Land_2019}, \cite[Definition~11.5]{Weibel2013TheKA}). 
Moreover, one can show that the formula\footnote{I thank Sasha Efimov for pointing this out to me.}
\[
EH(\sC) = \colim\limits_{[n] \in \Delta^{\mrm{op}}} E(\sC \otimes_k \mbf{Perf}_{k[t]^{\otimes n}})
\]
yields an explicit description for $L_{\mathbb{A}^1}E$. 
Indeed, $EH$ defines a finitary localizing invariant and the map $E \to EH$ induces an equivalence upon 
applying $L_{\mathbb{A}^1}$, so it suffices to show that $EH$ is
$\mathbb{A}^1$-invariant. This follows from the following lemma.

\begin{lem}
The composite map
\[
\colim\limits_{\Delta^{\mrm{op}}} k[t]\otimes k[t]^{\otimes \bullet} \to \colim\limits_{\Delta^{\mrm{op}}}k[t]^{\otimes \bullet} \to \colim\limits_{\Delta^{\mrm{op}}} k[t]\otimes k[t]^{\otimes \bullet}
\]
is an equivalence of connective ring spectra. 
Consequently, the map $EH(\sC) \to EH(\sC \otimes_k \mbf{Perf}_{k[t]})$ is an equivalence.
\end{lem}
\begin{proof}
If $k$ is discrete or, more generally, $\mathbb{Z}$-linear, the proof can be found in 
\cite[IV.11.5.1(1)]{Weibel2013TheKA} or \cite[Proposition~5.2(1)]{tabuada_2015}.

In general, the forgetful functor $\mrm{CAlg}_k \to \mbf{Mod}_k$ is conservative (see 
\cite[Lemma~3.2.2.6]{HA}) and 
the extension of scalars functor $- \otimes_k \pi_0(k)$ is conservative on connective objects (it induces an 
equivalence on the lowest homotopy group), so we can 
reduce the general case to the discrete case.
\end{proof}

\begin{rem}
If $k$ is $\mathbb{Z}$-linear the simplicial diagram in the definition of 
$EH(\sC)$ is equivalent to the one in \cite[Definition~11.5]{Weibel2013TheKA}. However, over $\mathbb{S}$ 
we cannot use the presentation of $k[t]^{\otimes n}$ as the 
quotient of $k[t_0,\cdots, t_n]$ modulo the relation $t_0 + \cdots + t_n = 1$.
\end{rem}

\ssec{Smooth \& proper categories}

\begin{defn}
Given small stable $\infty$-categories $\sC$, $\sD$ a $\sC$-$\sD$-{\bf bimodule} is 
an exact functor 
\[
\sC^{op} \otimes_k \sD \to \Spt.
\]
For any $\sC$ there is a canonical $\sC$-$\sC$-bimodule given by the mapping spectrum functor, which we call 
canonical. 
A {\bf perfect} $\sC$-$\sD$-bimodule is a bimodule which is a compact object of the functor category 
$\mrm{Fun}^{\mrm{ex}}(\sC^{op} \otimes_k \sD, \Spt)$.
\end{defn}

\begin{defn}
For $\sC \in \mrm{Cat}^{\mrm{perf}}_{\infty,k}$ we say:
\begin{itemize}
\item $\sC$ is {\bf smooth} if $\sC$ is perfect as a $\sC$-$\sC$-bimodule.
\item $\sC$ is {\bf proper} if the mapping spectra in $\sC$ are perfect $k$-modules.
\end{itemize}
\end{defn}

\begin{rem}
The notions of smoothness and properness are evidently self-dual.
\end{rem}

\begin{rem}
Assume $k$ be a field. For a classical scheme $X \in \mrm{Sch}_k$ that is of finite type and separated,
being proper or smooth is equivalent to 
$\mbf{Perf}_X$ being proper or smooth (see \cite[Proposition~3.31]{Orlov_2016}; see also 
\cite[Example~0.4]{neeman_stronggen} for a more general statement).

\end{rem}

\begin{rem}
For $\sC, \sD \in \mrm{Cat}^{\mrm{perf}}_{\infty,k}$ such that $\sC$ is smooth and proper, we have
\[
\mrm{maps}_{\mathcal{M}_{\mrm{loc}, k}}(\mathcal{U}_{\mrm{loc}}(\sC), \mathcal{U}_{\mrm{loc}}(\sD)) \simeq K(\sC^{\mrm{op}} \otimes_k \sD).
\]
This is \cite[Theorem~5.25]{HSS} (see also \cite[Theorem~9.36]{blumberg2013universal}).
\end{rem}

\begin{lem}\label{lem:smooth_right_compact}
For $\sC, \sD \in \mrm{Cat}^{\mrm{perf}}_{\infty,k}$, such that $\sC$ is smooth, the functor 
\[
\mrm{Fun}^{\mrm{ex}}(\sC^{\mrm{op}}, \sD) \to \mrm{Fun}^{\mrm{ex}}(\sC^{\mrm{op}}, \widehat{\sD})
\] 
factors through the subcategory $\sC \otimes_k \sD$ (see \ref{ssec:sym_mon}).
\end{lem}
\begin{proof}
By naturality we have a commutative diagram
\[
\begin{tikzcd}
\mrm{Fun}^{\mrm{ex}}(\sC^{\mrm{op}}, \sC^{\mrm{op}})\arrow[d]\arrow[r] & \mrm{Fun}^{\mrm{ex}}(\sC^{\mrm{op}}, \widehat{\sC^{\mrm{op}}})\arrow[d] 
& \sC \otimes_k \sC^{\mrm{op}}\arrow[l]\arrow[d]\\
\mrm{Fun}^{\mrm{ex}}(\sC^{\mrm{op}}, \sD) \arrow[r] & \mrm{Fun}^{\mrm{ex}}(\sC^{\mrm{op}}, \widehat{\sD})
& \sC \otimes_k \sD\arrow[l]
\end{tikzcd}
\]
for any functor $\sC^{\mrm{op}} \to \sD$. Hence it suffices to show the claim for $\sD = \sC^{\mrm{op}}$ and 
the identity functor.
In other words, we need to show that the image of the identity functor along
\[
\mrm{Fun}^{\mrm{ex}}(\sC^{\mrm{op}}, \sC^{\mrm{op}}) \to 
\mrm{Fun}^{\mrm{ex}}(\sC^{\mrm{op}}, \widehat{\sC^{\mrm{op}}})\simeq \mrm{Fun}^{\mrm{ex}}(\sC^{op} \otimes_k \sC, \Spt)
\]
is compact. The image is the canonical bimodule which is compact by smoothness of $\sC$.
\end{proof}

\section{Endomorphism categories}\label{section:endcats}
\begin{defn}\label{defn:endomorphisms}
Let $\sC$ be a small $k$-linear idempotent complete stable $\infty$-category.
We introduce the following $\infty$-categories:
\begin{enumerate}
\item The 
stable $\infty$-category of endomorphisms on objects of $\sC$: 
\[
\mathcal{E}nd(\sC) : = \mrm{Fun}(B\mathbb{N}, \sC).
\]
\item The $\infty$-category of nilpotent endomorphisms
$\mathcal{N}il(\sC)$ is defined to be the full subcategory of $\mathcal{E}nd(\sC)$ consisting of endomorphisms 
$x \stackrel{p}\to x$ such that the colimit in $\widehat{\sC}$ 
\[
x[p^{-1}] := \colim (x \stackrel{p}\to x \stackrel{p}\to x \stackrel{p}\to \cdots)
\]
is trivial. 
It is a stable subcategory because the colimit functor 
\[
\mrm{Fun}(B\mathbb{N}, \sC) \to \mrm{Fun}(B\mathbb{N}, \widehat{\sC}) \stackrel{\colim} \to \widehat{\sC}
\]
is exact.
\end{enumerate}
\end{defn}

\begin{lem}\label{lem:modules_vs_end}
For a small $k$-linear idempotent complete stable $\infty$-category $\sC$ we have a natural equivalence
\[
\mrm{Fun}^{\mrm{ex}}_k(\mbf{Perf}_{k[t]}, \sC) \simeq  \mrm{Fun}(B\mathbb{N}, \sC).
\]
\end{lem}
\begin{proof}
For $k=\mathbb{S}$ the claim follows from \cite[Proposition~3.9]{blumberg2015witt}. 
The general result follows from the adjunction (see \cite[Remark~4.5.3.2]{HA})
\[
\mrm{Fun}^{\mrm{ex}}_k(\mbf{Perf}_{\mathbb{S}[t]} \otimes_{\mathbb{S}} \mbf{Perf}_{k}, \sC) \simeq 
\mrm{Fun}^{\mrm{ex}}(\mbf{Perf}_{\mathbb{S}[t]}, \sC)
\]
and the equivalence $\mbf{Perf}_{\mathbb{S}[t]} \otimes_{\mathbb{S}} \mbf{Perf}_{k} = \mbf{Perf}_{k[t]}$ (see \cite[Remark~4.8.5.17]{HA}).
\end{proof}

\begin{lem}\label{lem:dual_desc}
There is a canonical equivalence between 
$\mathcal{E}nd(\sC)$ 
and the full subcategory 
of
\[
\mrm{Fun}_k^{\mrm{ex}}(\sC^{\mrm{op}}, \mbf{Mod}_{k[t]})
\]
consisting of  functors $F:\sC^{\mrm{op}} \to \mbf{Mod}_{k[t]}$ for which the composite 
\[
\sC^{\mrm{op}} \to \mbf{Mod}_{k[t]} \stackrel{\mathcal{U}}\to \mbf{Mod}_{k}
\]
is representable.
\end{lem}
\begin{proof}
By Lemma~\ref{lem:modules_vs_end} we have an equivalence 
\[
\mathcal{E}nd(\sC) \simeq \mrm{Fun}^{\mrm{ex}}_k(\mbf{Perf}_{k[t]}, \sC).
\]
Under this identification taking the underlying object of 
an endomorphism corresponds to 
evaluating a functor on the tautological module $k[t]$-module $k[t]$. 
Now consider the fully faithful composite 
\[
\Psi: \mrm{Fun}^{\mrm{ex}}_k(\mbf{Perf}_{k[t]}, \sC) \to \mrm{Fun}^{\mrm{ex}}_k(\mbf{Perf}_{k[t]}, \widehat{\sC}) \simeq \mrm{Fun}^{\mrm{ex}}_k(\mbf{Perf}_{k[t]} \otimes_k \sC^{\mrm{op}}, \mbf{Mod}_k) \simeq \mrm{Fun}^{\mrm{ex}}_k(\sC^{\mrm{op}}, \mbf{Mod}_{k[t]}).
\]
It sends an exact functor $F:\mbf{Perf}_{k[t]}\to \sC$ to a functor 
\[
\sC^{\mrm{op}} \to \mbf{Mod}_{k[t]}
\] 
\[
X \mapsto \mrm{maps}_\sC(X, F(k[t])).
\]
Here the action of $k[t]$ is induced by the action on $F(k[t])$. 
Forgetting the action we get a functor represented by $F(k[t])$. On the other hand, for any exact functor
$\sC^{\mrm{op}} \stackrel{G}\to \mbf{Mod}_{k[t]}$ for which the composite
\[
\sC^{\mrm{op}} \stackrel{G}\to \mbf{Mod}_{k[t]} \stackrel{\mathcal{U}}\to \mbf{Mod}_{k}
\]
is representable by $Y$ in $\sC$, consider the functor 
\[
\tilde{G}:\mbf{Perf}_{k[t]} \to \sC
\]
sending 
$k[t]$ to $Y$ with the induced action of $k[t]$. Now by construction $\Psi(\tilde{G}) = G$.
\end{proof}

The usage of the word {\it nilpotent} in Definition~\ref{defn:endomorphisms}(2) is not accidental: 

\begin{lem}\label{lem:nilp}
For $(P,\phi) \in \mathcal{N}il(\sC) \subset \mathcal{E}nd(\sC)$ there exists $n\in \mathbb{N}$ such that 
$\phi^n \simeq 0$.
\end{lem}
\begin{proof}
In view of the exact sequences
\[
\Fib(P \stackrel{\phi^n}\to P) \to P \stackrel{\phi^n}\to P
\]
and the vanishing of $P[\phi^{-1}]$ we have 
\[
P = \colim\limits_n \Fib(P \stackrel{\phi^n}\to P).
\]
Now by compactness of $\sC \subset \widehat{\sC}$, we see that there exists an $n\in\mathbb{N}$ and a 
map $P \to \Fib(P \stackrel{\phi^n}\to P)$ making 
the following commutative diagram commute:
\[
\begin{tikzcd}
P \arrow[r]\arrow[dr, "\id"] & \Fib(P \stackrel{\phi^n}\to P)\arrow[d]\arrow[dr, "0"] \\
& P\arrow[r,"\phi^n"] &P.
\end{tikzcd}
\]
In particular $\phi^n \simeq 0$.
\end{proof}

\begin{rem}\label{rem:end_vs_ktmod}
Since $\mbf{Perf}_{k[t]}$ is smooth, by Lemma~\ref{lem:smooth_right_compact} we have a fully faithful embedding
\[
\mathcal{E}nd(\sC) \to \sC \otimes_k \mbf{Perf}_{k[t]}.
\]
We also have an embedding 
\[
\mathcal{E}nd(\sC) \to \mrm{Fun}_k^{\mrm{ex}}(\sC^{\on{op}}, \mbf{Mod}_{k[t]}).
\]
\end{rem}

\ssec{Tensoring endomorphisms}\label{ssec:action}
Note that 
$\mrm{Fun}(B\mathbb{N}, -) : \mrm{Cat}^{\mrm{perf}}_{\infty, k} \to \mrm{Cat}^{\mrm{perf}}_{\infty, k}$ is 
lax symmetric monoidal, so the symmetric monoidal structure on $\mbf{Perf}_k$ induces a sectionwise symmetric monoidal structure on 
\[
\mathcal{E}nd(\mbf{Perf}_k) = \mrm{Fun}(B\mathbb{N}, \mbf{Perf}_k).
\]
Now for $\sC$ being a small $k$-linear stable $\infty$-category  $\mathcal{E}nd(\sC)$ is 
$\mathcal{E}nd(\mbf{Perf}_{k})$-tensored. In particular, the biexact pairing 
\[
\otimes: \sC \times \mbf{Perf}_{k} \to \sC
\]
induces a biexact pairing
\[
\otimes : \mathcal{E}nd(\sC) \times \mathcal{E}nd(\mbf{Perf}_k) \to \mathcal{E}nd(\sC)
\]
acting as follows:
\[
(P,\phi), (Q,\psi) \mapsto (P\otimes Q, \phi \otimes \psi).
\]
This also restricts to an action of 
$\mathcal{E}nd(\mbf{Perf}_{k})$ on $\mathcal{N}il(\sC)$. 

\ssec{}\label{ssec:nil0}
Sending an object $P \in \sC$ to the pair $(P,0)$ gives rise to an exact 
functor
\[
\iota:\sC \to \mathcal{E}nd(\sC)
\]
which factors through an exact 
 functor
\[
\iota_{\on{nil}}:\sC \to \mathcal{N}il(\sC).
\]
This functor admits a left inverse 
\[
\mathcal{E}nd(\sC) \to \sC
\]
given by sending an endomorphism to its underlying module.  
We denote the respective cofibers of the maps 
\[
\mathcal{U}_{\on{loc},k}(\sC) \to \mathcal{U}_{\on{loc},k}(\mathcal{E}nd(\sC))
\] 
\[
\mathcal{U}_{\on{loc},k}(\sC) \to \mathcal{U}_{\on{loc},k}(\mathcal{N}il(\sC))
\] 
in the $\infty$-category $\mathcal{M}_{\mrm{loc},k}$ by $\mathcal{E}nd_0(\sC)$ and $\mathcal{N}il_0(\sC)$. 
We also may think of them as the images of idempotents 
\[
\pi: (P,\phi) \mapsto (P,\phi) - (P,0).
\]
Abusing the notation, we denote by
$E(\mathcal{E}nd_0(\sC))$ and $E(\mathcal{N}il_0(\sC))$ 
the images of the corresponding idempotents on $E(\mathcal{E}nd(\sC))$ and $E(\mathcal{N}il(\sC))$ 
for any $k$-localizing invariant $E$ (i.e. not necessarily finitary). 
This is justified by the fact that for a finitary $E$ we have
\[
E(\mathcal{E}nd_0(\sC)) = \mrm{maps}(\mathcal{E}nd_0(\sC), E),
\]
\[
E(\mathcal{N}il_0(\sC)) = \mrm{maps}(\mathcal{N}il_0(\sC), E).
\]

\begin{rem}
The image of $K_0(\sC)$ in $K_0(\mathcal{E}nd(\sC))$ gives an ideal $\mathcal{I}_{\on{triv}}$, so
$K_0(\mathcal{E}nd(\sC))$ and $K_0(\mathcal{E}nd_0(\sC))$ are both rings. 
\end{rem}

\section{Nilpotent endomorphisms obstruct \texorpdfstring{$\mathbb{A}^1$}{A1}-invariance}\label{section:nilend_a1inv}

Most of the paper is dedicated to studying the $\infty$-categories $\mathcal{N}il(\sC)$ and 
$\mathcal{E}nd(\sC)$. 
This section explains the main reason we need to do it: as Proposition~\ref{prop:role_of_nil} shows, 
the $\infty$-category $\mathcal{N}il(\sC)$ captures the failure of $\mathbb{A}^1$-invariance. 

\begin{defn}
Let $E$ be a functor 
\[
\mrm{Cat}^{\mrm{perf}}_{\infty, k} \to \Spt.
\]
For any $\sC \in \mrm{Cat}^{\mrm{perf}}_{\infty, k}$ we define 
\[
NE(\sC) := \Fib(E(\sC \otimes_k \mbf{Perf}_{k[t]}) \to E(\sC  \otimes_k \mbf{Perf}_{k})),
\]
where $k[t] \to k$ is the evaluation map. 
\end{defn}

\sssec{} $NE(-)$ promotes to a functor $\mrm{Cat}^{\mrm{perf}}_{\infty, k} \to \Spt$. 
Equivalently $NE(\sC)$ can be defined as the cofiber of the map induced by the $k$-algebra structure:
\[
E(\sC  \otimes_k \mbf{Perf}_{k}) \to E(\sC  \otimes_k \mbf{Perf}_{k[t]}).
\]

\begin{exam}
The nonconnective K-theory spectrum gives rise to a functor $\mrm{Cat}^{\mrm{perf}}_{\infty, k} \to \Spt$.
We have $NK(\mbf{Perf}_X) = 0$ if $X$ is a regular noetherian scheme, however, in general
it is non-trivial. 
\end{exam}

\begin{exam}
The functor $L_{\mathbb{A}^1}E$ defined in \ref{rem:A1localization} satisfies $N(L_{\mathbb{A}^1}E)(\sC) \simeq 0$ for all $\sC \in \mrm{Cat}^{\mrm{perf}}_{\infty, k}$.
\end{exam}

\begin{prop}\label{prop:role_of_nil}
For a small $k$-linear idempotent complete stable $\infty$-category $\sC$ 
the canonical embedding from Remark~\ref{rem:end_vs_ktmod} induces an equivalence 
\[
\Psi:\mathcal{N}il(\sC) \to \mrm{Ker}(\sC \otimes_k \mbf{Perf}_{k[t]} \to \sC \otimes_k \mbf{Perf}_{k[t,t^{-1}]}).
\]
Futhermore, the map $\iota_{\mrm{nil}}$ corresponds to the functor 
\[
\sC \otimes_k ev_{*} : \sC \otimes_k \mbf{Perf}_k \to \sC \otimes_k \mbf{Perf}_{k[t]}
\]
under this equivalence. Here $ev_{*}$ is the restriction of scalars functor along the evaluation map.
\end{prop}
\begin{proof}
$\Psi$ is fully faithful by construction, so we need to show that for any 
\[
(P, \phi) \in (\mrm{Fun}^\mrm{ex}(B\mathbb{N}, \widehat{\sC}))^\omega \simeq \sC \otimes_k \mbf{Perf}_{k[t]}
\]
such that $P[\phi^{-1}] \simeq 0$ we have that $P$ belongs to $\sC= (\widehat{\sC})^\omega$. 

First assume that $\phi \simeq 0$. 
Consider the functor 
\[
\widehat{\sC} \to \mrm{Fun}^\mrm{ex}(B\mathbb{N}, \widehat{\sC})
\]
\[
X \mapsto (X,0).
\]
It preserves colimits, hence its right adjoint functor 
\[
\mrm{Fun}^\mrm{ex}(B\mathbb{N}, \widehat{\sC}) \to \widehat{\sC}
\]
\[
(X,\phi) \mapsto \Cofib(\phi)
\]
preserves compact objects. Hence 
$\Cofib(P \stackrel{0}\to P) = P \oplus \Sigma P$ is compact, and so is $P$, as its retract.

In general, by Lemma~\ref{lem:nilp} there exists an $n\in \mathbb{N}$ such that $\phi^n \simeq 0$. 
The pushforward along the map of rings $t \mapsto t^n$ yields a functor  
\[
\mbf{Perf}_{k[t]} \to \mbf{Perf}_{k[t]}.
\]
This gives a functor 
\[
\sC \otimes_k \mbf{Perf}_{k[t]} \to \sC \otimes_k \mbf{Perf}_{k[t]} 
\]
which sends $(P,\phi)$ to $(P,\phi^n) \simeq (P,0)$. 
By the above reasoning, $P$ belongs to $\sC$.

The second part of the claim follows from the construction.
\end{proof}

\begin{cor}\label{cor:nil_vs_nilk}
There is an equivalence
$\mathcal{N}il(\sC) \simeq \sC \otimes_k \mathcal{N}il(\mbf{Perf}_{k})$.
\end{cor}
\begin{proof}
By Proposition~\ref{prop:role_of_nil} we have an exact sequence 
\[
\mathcal{N}il(\mbf{Perf}_{k}) \to \mbf{Perf}_{k[t]} \to \mbf{Perf}_{k[t,t^{-1}]}
\]
which by Lemma~\ref{lem:tensor_prod_ex} induces an exact sequence
\[
\sC \otimes_k \mathcal{N}il(\mbf{Perf}_{k}) \to \sC \otimes_k \mbf{Perf}_{k[t]} \to \sC \otimes_k 
\mbf{Perf}_{k[t,t^{-1}]}.
\]
Applying Proposition~\ref{prop:role_of_nil} again we get the desired equivalence. 
\end{proof}

\begin{prop}\label{prop:NE=ENil}
For a $k$-localizing invariant $E$ we have a natural equivalence
\[
\Omega NE(\sC) \simeq E(\mathcal{N}il_0(\sC))
\]
of functors $\mrm{Cat}^{\mrm{perf}}_{\infty, k} \to \Spt$.
\end{prop}
\begin{proof}
Denote the projection maps  $\mathbb{P}^1_k \to \mrm{Spec}(k)$  and $\mathbb{A}^1_k \to \mrm{Spec}(k)$ by $pr_{\mathbb{P}^1}$ and $pr_{\mathbb{A}^1}$, respectively. 
We have maps given by embedding $0$ and its complement into $\mathbb{P}^1_k$:
\[
\mrm{Spec}(k) \stackrel{i_0} \to \mathbb{P}^1_k \stackrel{j_{0}}\leftarrow  \mathbb{A}^1_k. 
\]
We also have an embedding given by the complement to the section at infinity: 
\[
\mathbb{A}^1_k \stackrel{j_\infty}\to \mathbb{P}^1_k.
\]
Consider the following diagram:
\[
\begin{tikzcd}
E(\sC \otimes_k \mbf{Perf}_{\mathbb{P}_k^1, \{0\}}) \arrow[d,"j"]\arrow[r]& E(\sC \otimes_k \mbf{Perf}_{\mathbb{P}_k^1}) \arrow[d, "j_{\infty}^*"]\arrow[r,"j_0^*"] & E(\sC \otimes_k \mbf{Perf}_{\mathbb{A}_k^1})\arrow[d]\\
E(\mathcal{N}il(\sC)) \arrow[r] & E(\sC \otimes_k \mbf{Perf}_{\mathbb{A}_k^1}) \arrow[r]          & E(\sC \otimes_k \mbf{Perf}_{\mathbb{G}_{m,k}}).
\end{tikzcd}
\]
By Proposition~\ref{prop:role_of_nil} the rows in diagram are fiber sequences. 
By Lemma~\ref{lem:tensor_prod_ex} $E(\sC \otimes_k -)$ is a $k$-localizing invariant, so the right square is a pullback square by Proposition~\ref{prop:P1_excision}. This implies that the 
left vertical map $j$ is an equivalence. 
It suffices to construct a horizontal map in the commutative diagram 
\[
\begin{tikzcd}
E(\sC) \arrow[dotted, r]\arrow[rd, swap, "\iota_{nil}"] & E(\sC \otimes_k \mbf{Perf}_{\mathbb{P}_k^1, \{0\}})\arrow[d, "j"]\\
& E(\mathcal{N}il(\sC))
\end{tikzcd}
\]
whose cofiber is $\Omega NE(\sC)$. 
It follows from the ``furthermore'' part of Proposition~\ref{prop:role_of_nil} that the direct image 
functor\footnote{note that $i_{0,*}$ is well defined because $i_0$ is a regular embedding}
$i_{0,*}:\sC\otimes_k \mbf{Perf}_k \to \sC\otimes_k \mbf{Perf}_{\mathbb{P}^1_k}$ makes the diagram commutative. So we only need to show that its cofiber $C$ is naturally equivalent to $\Omega NE(\sC)$.

The projective bundle formula of Proposition~\ref{prop:pbf} yields that 
the map 
\[
E(\sC\otimes_k \mbf{Perf}_k) \oplus E(\sC\otimes_k \mbf{Perf}_k) \stackrel{\begin{pmatrix}i_{0,*} & pr^*_{\mathbb{P}^1} \end{pmatrix}}\to E(\sC \otimes_k \mbf{Perf}_{\mathbb{P}^1_k})
\]
is an equivalence. 
To finish the proof, observe that the following diagram is commutative
\[
\begin{tikzcd}[ampersand replacement=\&]
E(\sC \otimes_k \mbf{Perf}_k)\arrow{r}{\begin{pmatrix} 1\\ 0\end{pmatrix}}\arrow[d,"i_{0,*}"] \& 
E(\sC \otimes_k \mbf{Perf}_k) \oplus E(\sC \otimes_k \mbf{Perf}_k)\arrow{r}{\begin{pmatrix} 0 & 1\end{pmatrix}}\arrow{d}{\begin{pmatrix} i_{0,*} & pr_{\mathbb{P}^1}^* \end{pmatrix}} \&
E(\sC \otimes_k \mbf{Perf}_k)\arrow[d,"pr^*_{\mathbb{A}^1}"]\\
E(\sC \otimes_k \mbf{Perf}_{\mathbb{P}^1_k, \{0\}})\arrow[r]\arrow[d] \&
E(\sC \otimes_k \mbf{Perf}_{\mathbb{P}^1_k})\arrow[r, "j_0^*"]\arrow[d] \&
E(\sC \otimes_k \mbf{Perf}_{\mathbb{A}^1_k})\arrow[d]\\
C \arrow[r] \& 0 \arrow[r] \& NE(\sC)
\end{tikzcd}
\]
and all its rows and columns are fiber sequences.
\end{proof}

\begin{cor}\label{cor:action_W0}
Let $E$ be any $k$-localizing invariant and $\sC$ be a small $k$-linear stable $\infty$-category. 
Then $\pi_iNE(\sC)$ admits an action of the ring $K_0(\mathcal{E}nd_0(\mbf{Perf}_k))$ for any $i\in \mathbb{Z}$.  
%
%
%
%
\end{cor}
\begin{proof}
By Proposition~\ref{prop:NE=ENil} $\pi_iNE(\sC) \simeq E_{i-1}(\mathcal{N}il_0(\sC))$. 
Thus it suffices to construct the action of $K_0(\mathcal{E}nd_0(\mbf{Perf}_k))$ on 
$\pi_iE(\mathcal{N}il_0(\sC))$.
The action of $\mathcal{E}nd(\mbf{Perf}_k)$ on $\mathcal{N}il(\sC)$ (see \ref{ssec:action}) 
\[
\mathcal{E}nd(\mbf{Perf}_k) \otimes_k \mathcal{N}il(\sC)\to \mathcal{N}il(\sC)
\]
by adjunction yields a functor
\[
\mathcal{E}nd(\mbf{Perf}_k)^{\simeq} \to \Fun^{\mrm{ex}}(\mathcal{N}il(\sC),\mathcal{N}il(\sC))^{\simeq}= \mrm{End}_{\mrm{Cat}^{\mrm{perf}}_{\infty, k}}(\mathcal{N}il(\sC)) \to \mrm{End}_{\mathbb{Z}}(\pi_iE(\mathcal{N}il(\sC))).
\]
By additivity theorem this descents into a map
\[
K_0(\mathcal{E}nd(\mbf{Perf}_k)) \to \mrm{End}_{\mathbb{Z}}(\pi_iE(\mathcal{N}il(\sC)))
\]
which is a map of rings by construction. 
This action commutes with the idempotent given by sending $(P,\phi) \in \mathcal{N}il(\sC)$ to $(P,0)$ (see \ref{ssec:nil0}), so we 
in fact have a map of rings 
\[
K_0(\mathcal{E}nd(\mbf{Perf}_k)) \to \mrm{End}_{\mathbb{Z}}(\pi_iE(\mathcal{N}il_0(\sC))).
\]
Finally, the ideal $\mathcal{I}_{\on{triv}}$ is sent to $0$, so we obtain the desired action
\[
K_0(\mathcal{E}nd_0(\mbf{Perf}_k)) \to \mrm{End}_{\mathbb{Z}}(\pi_iE(\mathcal{N}il_0(\sC))).
\]

%
%
%
%
\end{proof}

\section{Witt vectors}\label{section:characteristic polynomial}
Assume $k$ is a discrete ring. One can check that the map 
\[
\mathcal{E}nd(k) \to (1 + tk[\![t]\!])^{\times}
\]
given by sending a pair $(P,\phi) \in \mathcal{E}nd(k)$ to its {\bf characteristic polynomial} $\mrm{det}(1-t\phi)$
is a group homomorphism. Here $\mathcal{E}nd(k)$ denotes the exact category of endomorphisms on finitely 
generated projective $k$-modules. 
In fact, the codomain of the map can be endowed with a ring structure in such a way 
that this map 
is also a ring homomorphism. This ring is called the 
{\bf ring of Witt vectors} and is denoted by $\mathbb{W}(k)$. A trivial endomorphism has 
trivial characteristic polynomial, so we end up having a ring homomorphism
\[
K_0(\mathcal{E}nd_0(k)) \to \mathbb{W}(k).
\] 
The image of this map consists of fractions of the form
\[
\frac{1 + a_1t + \cdots a_nt^n}{1 + b_1t + \cdots b_lt^l}.
\]
We call the subring spanned by the image the {\bf ring of rational Witt 
vectors} and denote by $\mathbb{W}_0(k)$. 

If $k$ is an arbitrary connective ring spectrum and $\mathcal{E}nd(\mbf{Perf}_k)$ is considered instead of 
$\mathcal{E}nd(k)$, the same story goes through, except that now the 
characteristic polynomial lands in power series over $\pi_0(k)$. 
The following is, in essence, a classical result of Almkvist \cite{ALMKVIST1974375}. As stated, however, it 
follows from \cite[Theorem~1.12]{blumberg2015witt} (see also \cite[Theorem~1.9]{blumberg2015witt}):

\begin{thm}[Almkvist]\label{thm:char_polynomial}\label{thm:WittKorelation}
The characteristic polynomial map 
\[
K_0(\mathcal{E}nd_0(\mbf{Perf}_k)) \to \mathbb{W}(\pi_0 (k))
\]
is an isomorphism onto its image $\mathbb{W}_0(\pi_0 (k))$. Moreover, it becomes an isomorphism after completing with respect to the topology defined by preimages 
of the ideals $(1 + t^N\pi_0(k)[\![t]\!])$. 
\end{thm}

\begin{exam}
$\mathbb{W}(\mathbb{F}_p)$ is isomorphic to the countable product of copies of $\mathbb{Z}_p$ (see 
\cite[Proposition~10]{HesselholtLectureNO}).
\end{exam}

One of the key properties of the ring of Witt vectors is the following:

\begin{prop}\label{prop:invert_witt}
Assume $R$ is a discrete $\mathbb{Z}[\frac{1}{l}]$-algebra for $l \in \mathbb{Z}$. 
Then $l$ is also invertible in $\mathbb{W}(R)$.
\end{prop}
\begin{proof}
See \cite[Proposition~1.2]{weibel-nk}.
\end{proof}

\ssec{Verschiebung and Frobenius}
Any element $x$ of the ring $\mathbb{W}(\pi_0(k))$ 
 can be represented as an infinite product 
\[
\prod\limits_{i\in\mathbb{N}}(1 - a_it).
\]
We define 
\[
V_n(1-at) = 1-at^n
\]
\[
F_n(1-at) = 1-a^nt.
\]
By additivity and continuity this extends to natural additive Verschiebung and Frobenius operations
\[
V_n, F_n : \mathbb{W}(\pi_0(k)) \to \mathbb{W}(\pi_0(k))
\]
for all $n\in \mathbb{N}$ (cf. \cite{Grayson78Witt}).

These operations have been studied thoroughly over the past century and they have been shown to satisfy many 
interesting properties (see \cite{HesselholtLectureNO}, \cite{stienstra1982operations}). 
These operations restrict to the subring $\mathbb{W}_0(\pi_0(k))$ and we use the same notation for the 
restricted operations. 
For our purposes the most important property is the following:


\begin{lem}\label{lem:verfrob}
Assume $R$ is a discrete $\mathbb{F}_p$-algebra. 
Then the equality 
\[
V_{p^l}F_{p^l}(x) = p^lx
\]
holds for all $x \in \mathbb{W}(R)$.
\end{lem}
\begin{proof}
By continuity and additivity it suffices to show the equality for $x=(1-at)$. 
In this case this is obvious from the definitions and the fact that taking the 
$p^l$-th power on $R$ is additive. 
\end{proof}

\begin{rem}\label{rem:versch_ideals}
One has a nice description of the topology on $\mathbb{W}(R)$ from Theorem~\ref{thm:WittKorelation} 
using the Verschiebung operations. 
Recall that any element $x\in \mathbb{W}(R)$ can be uniquely represented as a product 
\[
\prod\limits_{i\in\mathbb{N}}(1 - a_it^i) = \prod\limits_{i\in\mathbb{N}}V_i(1 - a_it).
\]
If $x \in 1+t^NR[\![t]\!]$ then $a_i = 0$ for $i<N$. Hence $x$ is a (possibly infinite) sum of $V_i(x_i)$ 
for $x_i \in \mathbb{W}(R)$ and $i\ge N$. Conversely, any such sum belongs to the ideal, so we have: 
\[
1+t^NR[\![t]\!] = \{\prod\limits_{i\ge N} V_i(x_i) | x_i\in \mathbb{W}(R)\}.
\]
\end{rem}

\ssec{Categorical Frobenius \& Verschiebung}
A crucial observation about the operations $V_n$ and $F_n$ is that they 
are induced by explicit endofunctors of $\mathcal{E}nd(k)$. 

\sssec{} For a map of commutative ring spectra $f: R\to S$ one has an adjunction
\[
\begin{tikzcd}
    \mbf{Mod}_S \arrow[r, shift left=1ex, "f_*"] & \mbf{Mod}_R\arrow[l, shift left=.5ex, "f^!"{}]
\end{tikzcd}
\]
The functor $f^!$ is called the {\bf coinduction} functor. It can be described explicitly as follows: 
\[
f^!(M) := \mrm{maps}_R(S,M)
\]
If $S$ is perfect over $R$ (for example, it is finite and flat), both $f_*$ and $f^!$ restrict to the 
subcategories of perfect modules and we get an adjunction 
\[
\begin{tikzcd}
    \mrm{Perf}_S \arrow[r, shift left=1ex, "f_*"] & \mrm{Perf}_R\arrow[l, shift left=.5ex, "f^!"{}].
\end{tikzcd}
\]

\begin{rem}\label{rem:p!adj}
If $S$ if perfect over $R$, then the functor $f^!$ commutes with colimits. 
This follows from the equivalences 
\[
f^! \simeq \mrm{maps}_{R}(S, P) \simeq S^\vee \otimes_{R} P
\]
that make $f^!$ a left adjoint.
\end{rem}

\begin{constr}\label{constr:frob_and_ver}
Fix a positive natural number $l$. Consider the $k$-linear map of algebras $p_l: k[t] \to k[t]$
that sends $t$ to $t^l$. 
The functors $p_{l,*}$ and $p_l^!$ induce functors 
\[
\mrm{Fun}^{\mrm{ex}}_k(\sC^{\mrm{op}}, \mbf{Mod}_{k[t]}) \to \mrm{Fun}^{\mrm{ex}}_k(\sC^{\mrm{op}}, \mbf{Mod}_{k[t]}).
\]
By Lemma~\ref{lem:dual_desc} we may identify $\mathcal{E}nd(\sC)$ as a subcategory $\mrm{Fun}^{\mrm{ex}}_k(\sC^{\mrm{op}}, \mbf{Mod}_{k[t]})$ of functors for which the composite
\[
\sC^{\mrm{op}} \to \mbf{Mod}_{k[t]} \stackrel{\mathcal{U}}\to \mbf{Mod}_k
\]
is representable.
Now $p_{l,*}$ and $p_l^!$ preserve this subcategory because $p_l$ is finite and free. 
We denote the restriction of $p_{l,*}$ to $\mathcal{E}nd(\sC)$ by $\mathcal{F}_l$ and the restriction of 
$p_l^!$ by $\mathcal{V}_l$. We call these functors the {\it categorical} Frobenius and Verschiebung.
\end{constr}

\begin{rem}\label{rem:explicit_frob_and_ver}
Unravelling the definitions we can describe the functors from Construction~\ref{constr:frob_and_ver} 
explicitly as follows. 
Given any $(P,\phi) \in \mathcal{E}nd(\sC)$ we have
\begin{itemize}

\item $\mathcal{F}_l(P,\phi) \simeq (P,\phi^l)$

\item
$\mathcal{V}_l(P,\phi) \simeq \mrm{maps}_{k[t]}(k[t^{\frac{1}{l}}], P) \simeq (P^l, \begin{pmatrix}
0 & \cdots & 0 & \phi \\
1 & \cdots & 0 & 0 \\
\vdots & \ddots & \vdots & \vdots \\
0 & \cdots & 1 & 0
\end{pmatrix}).$
\end{itemize}
\end{rem}

\begin{prop}\label{prop:verfrob_as_functors}
For $\sC= \mbf{Perf}_k$ the induced maps 
\[
K_0(\mathcal{E}nd_0(\sC)) \to K_0(\mathcal{E}nd(\sC)) \stackrel{K_0(\mathcal{V}_n), K_0(\mathcal{F}_n)}\longrightarrow K_0(\mathcal{E}nd(\sC)) \to K_0(\mathcal{E}nd_0(\sC))
\]
are equal to the Verschiebung and Frobenius maps under the isomorphism from Theorem~\ref{thm:WittKorelation}. 
\end{prop}
\begin{proof}
By Remark~\ref{rem:explicit_frob_and_ver} the functors $\mathcal{F}_n$ and $\mathcal{V}_n$ are the 
same as those defined in Section~5 of \cite{blumberg2015witt}.  
The analogous functors on the exact category of endomorphisms of finitely generated projective modules over 
$\pi_0(k)$ were also 
defined by Grayson in \cite{Grayson78Witt} and they were shown to induce the Verschiebung and Frobenius maps 
on the corresponding summand of $K_0$. 
By Theorem~5.7 of \cite{blumberg2015witt} the operations are compatible with the operations defined by Grayson.
\end{proof}

\begin{rem}
For any integer $i$ applying $\pi_iE(-)$ to $\mathcal{F}_l$ and $\mathcal{V}_l$ we obtain Frobenius and 
Verschiebung maps for any localizing invariant $E$. When there is no confusion we use $F_l$ and $V_l$ to 
denote these maps.
\end{rem}

\sssec{}\label{diagrams}
Recall that 
\[
\mrm{Fun}^{\mrm{ex}}_k(\mbf{Perf}_{k[t]}, \mbf{Mod}_k) \simeq \mbf{Mod}_{k[t]}
\]
admits a sectionwise symmetric monoidal structure, extending the symmetric monoidal structure on 
$\mathcal{E}nd(\sC)$
(see (\ref{ssec:action})).
By construction, the functor $p_{l,*}$ is symmetric monoidal. 
and is also closed (i.e. preserves the internal hom). 
The adjoint functor $p_l^!$ is lax symmetric monoidal, i.e. 
there is a natural transformation 
\[
\mu_{X,Y}: p_l^!(X) \otimes p_l^!(Y) \to p_l^!(X \otimes Y).
\]
This yields a map 
\[
X \otimes p_l^!(Y) \stackrel{\eta_X \otimes \id_{p_l^!(Y)}}\to p_l^!(p_* (X)) \otimes p_l^!(Y) \stackrel{\mu_{p_* (X), Y}}\to p_l^!(p_{l,*}(X) \otimes Y)
\]
of functors $\mbf{Mod}_{k[t]} \times \mbf{Mod}_{k[t]} \longrightarrow \mbf{Mod}_{k[t]}.$
This specialises to a map
\[
\xi_{X,Y}: X \otimes \mathcal{V}_l (Y) \to \mathcal{V}_l (\mathcal{F}_l (X) \otimes Y)
\]
of functors $\mathcal{E}nd(\sC) \otimes \mathcal{E}nd(\mbf{Perf}_k) \to \mathcal{E}nd(\sC).$ Here we use the 
embedding $\mathcal{E}nd(\sC) \to \mrm{Fun}_\mrm{ex}(\sC^{\mrm{op}}, \mbf{Mod}_{k[t]})$ from Remark~\ref{rem:end_vs_ktmod}.

\begin{lem}\label{lem:proj_formula}
The following holds for $X \in \mathcal{E}nd(\sC)$ and $Y \in \mathcal{E}nd(\mbf{Perf}_k):$
\begin{enumerate}
\item 
We have a natural equivalence $\mathcal{F}_l (X \otimes Y) \simeq \mathcal{F}_l (X) \otimes \mathcal{F}_l (Y)$.
\item
The fiber of the map $\xi_{X,Y}: X \otimes \mathcal{V}_l(Y) \to \mathcal{V}_l(\mathcal{F}_l(X) \otimes Y)$ admits a finite filtration by functors $\mathcal{E}nd(\sC) \otimes \mathcal{E}nd(\mbf{Perf}_k) \to 
\mathcal{E}nd(\sC)$
that factor through $\iota : \sC \to \mathcal{E}nd(\sC)$. 
\end{enumerate}
\end{lem}
\begin{proof}
The first claim follows from monoidality of $p_{l,*}$.

To prove (2) we use the description of Remark~\ref{rem:explicit_frob_and_ver}. For 
$(X,\phi) \in \mathcal{E}nd(\sC)$ and $(Y,\alpha) \in \mathcal{E}nd(\mbf{Perf}_k)$ we may identify 
$X \otimes \mathcal{V}_l(Y)$ with 
\[
(X \otimes Y^l, 
\begin{pmatrix}
0 & \cdots & 0 & \phi \otimes \alpha \\
\phi \otimes \id_Y & \cdots & 0 & 0 \\
\vdots & \ddots & \vdots & \vdots \\
0 & \cdots & \phi \otimes \id_Y & 0
\end{pmatrix}
)
\]
and 
$\mathcal{V}_l(F_l(X) \otimes Y)$ with 
\[
(X \otimes Y^l, 
\begin{pmatrix}
0 & \cdots & 0 & \phi^l \otimes \alpha \\
\id_X \otimes \id_Y & \cdots & 0 & 0 \\
\vdots & \ddots & \vdots & \vdots \\
0 & \cdots & \id_X \otimes \id_Y & 0
\end{pmatrix}
).
\]
Moreover, the map $X \otimes \mathcal{V}_l(Y) \to \mathcal{V}_l(F_l(X) \otimes Y)$ may be identified with 
\[
\begin{pmatrix}
\phi^{l-1} \otimes \id_Y & \cdots & 0 & 0 \\
\vdots & \ddots & \vdots & \vdots \\
0 & \cdots & \phi\otimes \id_Y & 0\\
0 & \cdots & 0 & \id_X\otimes \id_Y
\end{pmatrix}.
\]
Denote the fiber of this map by $K$. The underlying module is identified with $\bigoplus\limits_{i=1}^{l-1}\Fib(\phi^{l-i}) \otimes Y$ and 
the endomorphism $\psi$ is 
\[
\begin{pmatrix}
0 & \cdots & 0 & 0\\
\phi \otimes \id_Y & \cdots & 0 & 0 \\
\vdots & \ddots & \vdots & \vdots \\
0 & \cdots & \phi \otimes \id_Y & 0
\end{pmatrix}.
\]
In particular, the map $\psi^l$ is canonically null on $K$ as an object of $\sC$. 
Recall that for any $(N,\alpha_N), (M,\alpha_M) \in \mathcal{E}nd(\sC)$ we have a fiber sequence 
\[
\Maps_{\mathcal{E}nd(\sC)}(N,M) \to \Maps_{\sC}(N,M) \stackrel{\alpha^*_N - \alpha_{M*}}\to \Maps_{\sC}(N,M).
\] 
Take $M = K$. By the above computation, $\psi^l$ acts trivially on the middle and the right terms 
of the fiber sequence, hence $\phi^{2l}$ acts trivially on the left term as well. Hence, we have a canonical 
identification $\phi^{2l} \simeq 0$ on $\Fib(\xi_{X,Y})$. 


By construction, $\Fib(\xi_{X,Y})$ is induced by an exact functor 
$\mbf{Mod}_{\mathbb{S}[t_1]} \otimes \mbf{Mod}_{\mathbb{S}[t_2]} \to \mbf{Mod}_{\mathbb{S}[t_3]}$ which itself 
is induced by tensoring with the $\mathbb{S}[t_1,t_2]$-$\mathbb{S}[t_3]$-bimodule 
$M_{\xi} \in \mbf{Perf}_{\mathbb{S}[t_1,t_2, t_3]}$. The above computation shows that $t_3$ acts nilpotently 
on $M_{\xi}$. 
Consider the subcategory $\sK \subset \mbf{Perf}_{\mathbb{S}[t_1,t_2,t_3]}$ of modules that are finite 
extensions of modules that are in the image of 
\[
i_{*} : \mbf{Perf}_{\mathbb{S}[t_1,t_2]} \to \mbf{Perf}_{\mathbb{S}[t_1,t_2,t_3]}.
\]
Equivalently, this is the smallest stable subcategory of $\mbf{Perf}_{\mathbb{S}[t_1,t_2,t_3]}$ containing 
the image of $i_*$. The inclusion of $\sK$ into the subcategory $\sK^{nil} \subset \mbf{Perf}_{\mathbb{S}[t_1,t_2,t_3]}$ 
of perfect modules where $t_3$ acts 
nilpotently identifies the latter with the retract-closure of $\sK$. Note that it also induces an isomorphism 
\[
\mathbb{Z} \simeq K_0(\sK) \to K_0(\sK^{nil}) \simeq \mathbb{Z},
\]
so by \cite[Theorem~2.1]{ThomasonSubcategories} it is an equivalence of categories. 
Now $M_{\xi}$ is a finite extension of modules of the form $i_*M_i$. Each $i_*M_i$ induces a functor 
\[
\mathcal{E}nd(\sC) \otimes \mathcal{E}nd(\mbf{Perf}_k) \to \sC \stackrel{\iota}\to \mathcal{E}nd(\sC),
\]
so the filtration on $M_{\xi}$
induces a required filtration on $\Fib(\xi_{X,Y})$. 
\end{proof}

\begin{rem}
Lemma~\ref{lem:proj_formula}(2) implies that $X \otimes \mathcal{V}_l(Y)$ and $\mathcal{V}_l(\mathcal{F}_l(X) \otimes Y)$ have the 
same classes in $K_0(\mathcal{E}nd_0(\sC))$. Thus it can be thought of as a categorical version of the relation proved
in \cite{stienstra1982operations} and later in \cite{Tabuada_Kleinian}. 
\end{rem}

\begin{cor}\label{cor:multn}
Assume $k$ is an $\mathbb{F}_p$-algebra. 
Let $\sC$ be a small $k$-linear idempotent complete stable $\infty$-category and let $E$ be a localizing invariant. 
Then the multiplication by $p^l \in K_0(\mathcal{E}nd_0(\mbf{Perf}_k))$ on $\pi_0E(\mathcal{E}nd_0(\sC))$ and 
on $\pi_0E(\mathcal{N}il_0(\sC))$ is equal to the composite
\[
\pi_0E(\mathcal{E}nd_0(\sC)) \to \pi_0E(\mathcal{E}nd(\sC)) \stackrel{\pi_0E(\mathcal{V}_{p^l}\circ \mathcal{F}_{p^l})}\to \pi_0E(\mathcal{E}nd(\sC)) \to \pi_0E(\mathcal{E}nd_0(\sC))
\]
and 
\[
\pi_0E(\mathcal{N}il_0(\sC)) \to \pi_0E(\mathcal{N}il(\sC)) \stackrel{\pi_0E(\mathcal{V}_{p^l}\circ \mathcal{F}_{p^l})}\to \pi_0E(\mathcal{N}il(\sC)) \to \pi_0E(\mathcal{N}il_0(\sC))
\]
respectively for any $l\in \mathbb{N}$.
\end{cor}
\begin{proof}
It follows from Lemma~\ref{lem:proj_formula}(2) that there is a fiber sequence 
\[
K(X,Y) \to  X \otimes \mathcal{V}_{p^l}(Y) \to \mathcal{V}_{p^l}(F_{p^l}(X) \otimes Y)
\]
where $K(X,Y)$ is a finite extension of functors that factor through $\iota: \sC \to \mathcal{E}nd(\sC)$ functorial in $X \in \mathcal{E}nd(\sC)$ and $Y \in \mathcal{E}nd(\mbf{Perf}_k)$.
Now we set $Y=(k,id)$. Since the $K_0$-class of $\mathcal{V}_{p^l}(k,id)$ is $p^l$ (this follows from 
Proposition~\ref{prop:verfrob_as_functors} and Lemma~\ref{lem:verfrob}), the endofunctor 
$X \mapsto X \otimes \mathcal{V}_{p^l}(k,id)$ induces the multiplication by $p^l$ on 
$\pi_0E$ of both $\mathcal{N}il(\sC)$ or $\mathcal{E}nd(\sC)$. On the other hand, by additivity 
$X \mapsto K(X,Y)$ induces the zero map on $\pi_0E$ of both $\mathcal{N}il_0(\sC)$ or $\mathcal{E}nd_0(\sC)$.
Hence, applying additivity again we see that $X \mapsto \mathcal{V}_{p^l}(\mathcal{F}_{p^l}(X))$ induces the multiplication by $p^l$ on $\pi_0E$ of either $\mathcal{N}il(\sC)$ or $\mathcal{E}nd(\sC)$.
\end{proof}

\section{Torsion in Nil-theories}\label{section:main_result}

This section is dedicated to proving the main results 
(Theorems~\ref{thm:main_result}~and~\ref{thm:main_result_2}). So far we 
never had to require our localizing invariants to be finitary. 
The reason we need it is to show that the action of $\mathbb{W}_0(\pi_0(k))$ 
constructed in Corollary~\ref{cor:action_W0} is continuous. This allows to extend the action 
to the action of the ring of Witt vectors $\mathbb{W}(\pi_0(k))$. The ring $\mathbb{W}(\pi_0(k))$
has many nice properties which the ring $\mathbb{W}_0(\pi_0(k))$ does not (for example, an integer 
$l$ is invertible in $\mathbb{W}(\pi_0(k))$ whenever it is invertible in $k$ by 
Proposition~\ref{prop:invert_witt}). 
The $\mathbb{W}(\pi_0(k))$-modules $NE(\sC)$ inherit these properties allowing to prove the main theorems. 

\ssec{}\label{ssec:truncated_polynomidal_rings}
We define $k[t]/(t^n)$ as the pushout of $\mathbb{E}_\infty$-$k$-algebras:
\[
\begin{tikzcd}
k[t] \arrow[r,"t \mapsto t^n"]\arrow[d, "t \mapsto 0"] & k[t]\arrow[d]\\
k \arrow[r] & k[t]/(t^n).
\end{tikzcd}
\]
By definition, it is a flat $k$-algebra and we have $\pi_\star(k[t]/(t^n)) \simeq \pi_\star(k)[t]/(t^n)$. 
We have canonical morphisms $k[t]/(t^n) \to k[t]/(t^m)$ for $m\le n$.

\ssec{}\label{ssec:filtration}
Now we 
introduce a filtration on the $\infty$-category $\mathcal{N}il(\sC)$. 
We start from the case $\sC = \mbf{Perf}_k$.  
Define $\mathcal{N}il^{\le n}(\mbf{Perf}_k)$ to be the stable $\infty$-category of those $k[t]/(t^n)$-modules 
that 
are perfect 
over $k$.
Restricting scalars gives rise to canonical functors 
\[
\tau_{n,m}:\mathcal{N}il^{\le n}(\mbf{Perf}_k) \to \mathcal{N}il^{\le m}(\mbf{Perf}_k) \text{ for any } 
n \le m 
\]
\[
\tau_n:\mathcal{N}il^{\le n}(\mbf{Perf}_k) \to \mathcal{N}il(\mbf{Perf}_k) \text{ for any } n.
\]


\begin{lem}\label{lem:filtration}
The canonical map 
\[
\colim\limits_n \mathcal{N}il^{\le n}(k) \to \mathcal{N}il(k)
\]
is an equivalence.
\end{lem}
\begin{proof}
The proof is the same as in \cite[Proposition~2.15]{BDM_K1local}. 
Note that $k[t]$ is assumed to be a discrete ring there: 
the assumption is needed to show that the map of pro-spectra
\begin{equation}\label{eq:pro-equiv}
k[t]/(t^n) \otimes_{k[t]} k[t]/(t^n) \to \{k[t]/(t^n) \otimes_{k[t]/(t^m)} k[t]/(t^n)\}_{m\ge n}.
\end{equation}
is a pro-equivalence. 
We show that it still holds in our setting. 

First, using the K\"{u}nneth spectral sequence, we 
can compute the homotopy groups of each of the modules in the tower. 
In particular, we see that the cofibers are free as $k$-modules, meaning that they are equivalent 
to infinite direct sums of copies of suspensions of $k$. 
Now the map 
\[
\begin{tikzcd}
\bigoplus \Sigma^{i_l} k \simeq \Cofib (k[t]/(t^n) \otimes_{k[t]/(t^{m+n})} k[t]/(t^n) \to k[t]/(t^n) \otimes_{k[t]} k[t]/(t^n))\arrow[d] \\
\Cofib (k[t]/(t^n) \otimes_{k[t]/(t^m)} k[t]/(t^n) \to k[t]/(t^n) \otimes_{k[t]} k[t]/(t^n))
\end{tikzcd}
\] 
is a zero map on the homotopy groups, and hence (by freeness) is a zero map. This shows the the 
cofiber of (\ref{eq:pro-equiv}) is a pro-null tower and that the map itself is a pro-equivalence. 

\end{proof}

\sssec{}\label{sssec:filtration}
For $\sC$ a small $k$-linear idempotent complete stable $\infty$-category we define 
\[
\mathcal{N}il^{\le n}(\sC) := \sC \otimes_k \mathcal{N}il^{\le n}(\mbf{Perf}_k).
\]
By Corollary~\ref{cor:nil_vs_nilk} we have an equivalence 
\[
\mathcal{N}il(\sC) \simeq \sC \otimes_k \mathcal{N}il(\mbf{Perf}_k)
\]
so we have a filtration of $\mathcal{N}il(\sC)$ by $\mathcal{N}il^{\le n}(\sC)$.

\begin{cor}\label{cor:filtration_general}
In the context above we have 
\[
\colim \mathcal{N}il^{\le n}(\sC) \simeq \mathcal{N}il(\sC). 
\]
\end{cor} 
\begin{proof}
Follows from  Lemma~\ref{lem:tensor_prod_ex} and Lemma~\ref{lem:filtration}.
\end{proof}

\begin{lem}\label{lem:Vpreservesnil}
The functor $\mathcal{V}_l \circ \iota_{\mrm{nil}}: \sC \to \mathcal{N}il(\sC)$ admits an $l$ step filtration such that the associated graded are equivalent to $\iota_{\mrm{nil}}$. 
In particular there is an equivalence of maps 
\[
\mathcal{V}_i \circ \iota_{\mrm{nil}} \simeq i\cdot\iota_{\mrm{nil}} : E(\sC) \to E(\mathcal{N}il(\sC))
\]
for any localizing invariant $E$.
\end{lem}
\begin{proof}
The base change theorem tells us that for all $l \in \mathbb{N}\setminus\{0\}$ we have commutative diagrams 
\[
\begin{tikzcd}
\mbf{Mod}_{k} \arrow[r,"u_l^!"]\arrow[d,"a_{0,*}"] & \mbf{Mod}_{k[t]/(t^l)}\arrow[d, "a_{l,*}"]\\
\mbf{Mod}_{k[t]} \arrow[r, "p^!_l"] & \mbf{Mod}_{k[t]}
\end{tikzcd}
\]
induced by the pushout diagrams of ring spectra
\[
\begin{tikzcd}
k[t] \arrow[d, "a_0"]\arrow[r, "p_l"] & k[t]\arrow[d, "a_l"]\\
k \arrow[r, "u_l"] & k[t]/(t^l)
\end{tikzcd}
\]
where the vertical maps are the obvious projections. 
By construction the composite $p^!_la_{0,*}$ induces 
\[
\mathcal{V}_l\circ \iota_{\mrm{nil}} : \mathcal{N}il(\sC) \to \mathcal{N}il(\sC).
\]
By base change isomorphism we have an equivalence 
\[
p^!_la_{0,*} \simeq a_{l}^!u_{l,*}
\]
and the composite functor $a_{l}^!u_{l,*}$ fits into an exact sequence 
\[
\begin{tikzcd}
a_{0,*}(-)\arrow[d,"\simeq"] \arrow[r]& a_{l}^!u_{l,*}(-)\arrow[d,"\simeq"] \arrow[r] & a_{l-1}^!u_{l-1,*}(-).\arrow[d,"\simeq"]\\
a_{0,*}(-) & a_{0,*}(-) \otimes_k k[t]/(t^l) & a_{0,*}(-) \otimes_k k[t]/(t^{l-1})
\end{tikzcd}
\]
which induces an exact sequences of functors 
\[
\sC \to \mathcal{N}il(\sC)
\]
\[
\iota_{\mrm{nil}} \to \mathcal{V}_l \circ \iota_{\mrm{nil}} \to \mathcal{V}_{l-1} \circ \iota_{\mrm{nil}}.
\]
By induction we obtain the desired filtration. 
Applying additivity we also see that 
\[
\mathcal{V}_l \circ \iota_{\mrm{nil}} : E(\sC) \to E(\mathcal{N}il(\sC)).
\]
is equivalent to $l$ copies of $\iota_{\mrm{nil}}$. 
\end{proof}

Now we are ready to state and prove our main result.

\begin{thm}\label{thm:main_result_proof}
Assume that $\pi_0(k)$ is an $\mathbb{F}_p$-algebra and $E$ is a finitary $k$-localizing invariant. 
Then for any small $k$-linear idempotent complete stable 
$\infty$-category $\sC$  
we have 
\[
NE(\sC)[\frac{1}{p}] \simeq 0.
\]
\end{thm}
\begin{proof}
It suffices to show that every element of $\pi_{i+1}NE(\sC) \simeq \pi_iE(\mathcal{N}il_0(\sC))$ 
is $p$-torsion 
(the equivalence is the content of Proposition~\ref{prop:NE=ENil}). 
In other words, 
by Corollary~\ref{cor:multn} it suffices to show that for every $x \in \pi_iE(\mathcal{N}il_0(\sC))$ 
there exists $l \in \mathbb{N}$ such that $(V_{p^l}\circ F_{p^l})(x)$ is in the image of
the map 
\[
\pi_iE(\iota_{\on{nil}}) : \pi_iE(\sC) \to \pi_iE(\mathcal{N}il(\sC)).
\]
Since $E$ is finitary and by Corollary~\ref{cor:filtration_general} we have
\[
\pi_iE(\mathcal{N}il(\sC)) \simeq \colim_l \pi_iE(\mathcal{N}il^{\le p^l}(\sC)).
\]
Hence any element $x \in \pi_iE(\mathcal{N}il(\sC))$ is the image of an 
element $\bar{x}$ of $\pi_iE(\mathcal{N}il^{\le p^l}(\sC))$. 
Restricting scalars along maps in the commutative diagram of ring spectra
\[
\begin{tikzcd}
k[t] \arrow[r, "p_l"]\arrow[d, "t \mapsto 0"] &k[t]\arrow[d]\\ 
k \arrow[r] & k[t]/(t^{p^l})
\end{tikzcd}
\]
we get a commutative diagram 

\[
\begin{tikzcd}
\mathcal{N}il^{\le p^l}(\sC) \arrow[r, "\tau_{p^l}"]\arrow[d, "\mathcal{U}"] & \mathcal{N}il(\sC)\arrow[d, "\mathcal{F}_{p^l}"]\\
\sC\arrow[r, "\iota_{\mrm{nil}}"] & \mathcal{N}il(\sC).
\end{tikzcd}
\]
Now 
\[
\mathcal{F}_{p^l}(x) = \mathcal{F}_{p^l}(\tau_{p^l} (\bar{x})) = 
\iota_{\mrm{nil}}(\mathcal{U}(\bar{x})).
\]
It suffices to show that $\mathcal{V}_{p^l}$ preserves the image of $\iota_{\mrm{nil}}$ which follows from Lemma~\ref{lem:Vpreservesnil}.
\end{proof}

\ssec{} 
Fix $E$ a finitary $k$-localizing invariant and $\sC$ a small $k$-linear idempotent complete stable 
$\infty$-category. 
Arguing as in the proof of Theorem~\ref{thm:main_result_proof} one can see that
$V_n(x)$ acts trivially on the image $\mbf{N}_n$ of $\pi_iE(\mathcal{N}il^{\le n}(\sC))$ in $\pi_iE(\mathcal{N}il_0(\sC))$  
for any $x\in \mathbb{W}_0(\pi_0(k))$. Indeed, for $y$ in the image we have an equality in 
$\pi_iE(\mathcal{N}il_0(\sC))$ (Lemma~\ref{lem:proj_formula} and Lemma~\ref{lem:Vpreservesnil}):
\[
y \cdot V_n(x) = V_n(F_n(y) \cdot x) = V_n(0 \cdot x) = 0.
\]
Now, any element of $\mathbb{W}(\pi_0(k))$ can be written as a product 
\[
\prod\limits_{i\in\mathbb{N}} V_i(x_i)
\]
for some $x_i \in \mathbb{W}_0(\pi_0(k))$ (see Remark~\ref{rem:versch_ideals}), hence we can make sense of 
the action 
\[
\mathbb{W}(\pi_0(k)) \times \mbf{N}_n \to \mbf{N}_n
\]
\[
(\prod\limits_{i\in\mathbb{N}} V_i(x_i), y) \mapsto \sum\limits_{i\in\mathbb{N}} V_i(x_i)\cdot y = \sum\limits_{0\le i\le n} V_i(x_i)\cdot y,
\]
where $x_i \in \mathbb{W}_0(\pi_0(k))$. This action is continuous and extends the action of 
$\mathbb{W}_0(\pi_0(k)) \subset \mathbb{W}(\pi_0(k))$. Moreover, 
\[
\pi_i(E(\mathcal{N}il_0(\sC))) \simeq \colim\limits_n \mbf{N}_n
\]
in the category of $\mathbb{W}_0(\pi_0(k))$-modules,
so summarizing we have:

\begin{lem}\label{lem:cont_action}
The action of Corollary~\ref{cor:action_W0} 
\[
\mathbb{W}_0(\pi_0(k)) \times \pi_i(E(\mathcal{N}il_0(\sC))) \to \pi_i(E(\mathcal{N}il_0(\sC)))
\]
is continuous and extends to an action of $\mathbb{W}(\pi_0(k))$. 
\end{lem}

Now we can prove our second main result.

\begin{thm}\label{thm:main_result_2_proof}
Assume $l$ is invertible in $\pi_0(k)$ for some $l\in\mathbb{Z}$. 
Let $E$ a finitary $k$-localizing invariant and $\sC$ a small $k$-linear idempotent complete stable 
$\infty$-category. Then
\begin{enumerate} 
\item 
$l$ is invertible in $\pi_iNE(\sC)$;
\item
$E/l$ is $\mathbb{A}^1$-invariant.
\end{enumerate}
\end{thm}
\begin{proof}
By Lemma~\ref{lem:cont_action} 
$\pi_iNE(\sC) \simeq \pi_{i-1}E(\mathcal{N}il_0(\sC))$ are continuous $\mathbb{W}(\pi_0(k))$-modules. 
Now the first claim follows from Proposition~\ref{prop:invert_witt}. 
The second claim follows from the equivalence 
\[
NE/l(\sC) \simeq N(E/l)(\sC).
\]
\end{proof}

\let\mathbb=\mathbf

{\small
\bibliography{references}
}

\parskip 0pt

\end{document}